\documentclass[12pt]{article}

\usepackage[dvipsnames]{xcolor}
\usepackage{graphicx}
\usepackage{amsmath,amsthm,amssymb}
\usepackage[normalem]{ulem} 

\usepackage[colorlinks=true,citecolor=black,linkcolor=black,urlcolor=blue]{hyperref}

\allowdisplaybreaks

\makeatletter
\g@addto@macro\normalsize{%
  \setlength\abovedisplayskip{8pt plus 3pt minus 3pt}
  \setlength\belowdisplayskip{8pt plus 3pt minus 3pt}
  \setlength\abovedisplayshortskip{6pt plus 3pt minus 2pt}
  \setlength\belowdisplayshortskip{6pt plus 3pt minus 2pt}
}
\makeatother

\normalsize

\definecolor{orange}{rgb}{0.9,0.45,0.1}
\definecolor{purple}{rgb}{0.85,0.33,0.83}

\numberwithin{equation}{section}

\def\nfrac#1#2{{\textstyle\frac{#1}{#2}}}
\let\leq\leqslant
\let\geq\geqslant

\newtheorem{thm}{Theorem}[section]
\newtheorem{cor}[thm]{Corollary}
\newtheorem{lemma}[thm]{Lemma}

\theoremstyle{definition}
\newtheorem{remark}[thm]{Remark}

\def\abs#1{\lvert#1\rvert} \let\card=\abs
\def\Abs#1{\bigl\lvert#1\bigr\rvert} 
\def\norm#1{\lVert#1\rVert}

\def\lambdamin{\lambda_{\mathrm{min}}}

\def\alphamax{\alpha_{\mathrm{max}}}
\def\Deltamax{\varDelta_{\mathrm{max}}}
\def\dfrac#1#2{\lower0.15ex\hbox{\large$\textstyle\frac{#1}{#2}$}}
\def\({\bigl(}
\def\){\bigr)}
\def\st{\mathrel{|}}
\def\St{\bigm|}

\def\One{\boldsymbol{1}}

\let\eps=\varepsilon

\def\P{\boldsymbol{P}}

\def\X{\boldsymbol{X}}
\def\Y{\boldsymbol{Y}}
\def\Z{\boldsymbol{Z}}

\def\calF{\mathcal{F}}
\def\calG{\mathcal{G}}

\def\calT{\mathcal{T}}
\def\F{\boldsymbol{\calF}}

\def\dvec{\boldsymbol{d}}
\def\hvec{\boldsymbol{h}}

\def\vvec{\boldsymbol{v}}
\def\xvec{\boldsymbol{x}}
\def\yvec{\boldsymbol{y}}
\def\zvec{\boldsymbol{z}}

\def\bad{{\mathrm{bad}}}
\def\good{{\mathrm{good}}}
\def\Gd{\mathcal{G}_{\dvec}}

\def\makesizes#1#2#3#4#5#6#7{%
\newcommand{#4}[1]{#2#1##1#3}
\newcommand{#5}[1]{#2\bigl#1##1\bigr#3}
\newcommand{#6}[1]{#2\Bigl#1##1\Bigr#3}
\newcommand{#7}[1]{#2\biggl#1##1\biggr#3}}

\def\E{\operatorname{\mathbb{E}}}
\makesizes [\E] \myE\mYE\MyE\MYE
\def\Var{\operatorname{Var}}
\makesizes [\Var] \myVar\mYVar\MyVar\MYVar
\def\Cov{\operatorname{Cov}}
\makesizes (\Cov) \myCov\mYCov\MyCov\MYCov
\def\Diam{\operatorname{diam}}
\makesizes [\Diam] \myDiam\mYDiam\MyDiam\MYDiam
\def\esssup{\operatorname{ess\,sup}\limits}
\makesizes [\esssup] \myesssup\mYesssup\Myessup\MYesssup

\def\Prob{\mathbb{P}}     
\def\V{\operatorname{\mathbb{V\!}}}

\def\Reals{{\mathbb{R}}}
\def\Complexes{{\mathbb{C}}}

\def\Integers{{\mathbb{Z}}}
\def\Aut{\operatorname{Aut}}
\def\Tilde#1{\mkern2mu\widetilde{\mkern-2mu#1}}

\def\nicebreak{\vskip 0pt plus 50pt\penalty-300\vskip 0pt plus -50pt }
\let\originalleft\left
\let\originalright\right
\renewcommand{\left}{\mathopen{}\mathclose\bgroup\originalleft}
\renewcommand{\right}{\aftergroup\egroup\originalright}

\begin{document}

\title{Subgraph counts for dense random graphs\\ with specified degrees\thanks{Research supported by Australian Research Council Discovery Project DP190100977.}}

\author{
Catherine Greenhill\\
\small School of Mathematics and Statistics\\[-0.8ex]
\small UNSW Sydney\\[-0.8ex]
\small NSW 2052, Australia\\[-0.3ex]
\small \texttt{c.greenhill@unsw.edu.au}\\
\and
Mikhail Isaev\\
\small School of Mathematics \\[-0.8ex]
\small Monash University\\[-0.8ex]
\small VIC 3800, Australia\\[-0.3ex]
\small\texttt{mikhail.isaev@monash.edu}
\and
Brendan D. McKay\\
\small Research School of Computer Science\\[-0.8ex]
\small Australian National University\\[-0.8ex]
\small ACT 2601, Australia\\[-0.3ex]
\small\tt brendan.mckay@anu.edu.au
}

\date{11 November 2020}

\maketitle

\begin{abstract}
We prove two estimates for the expectation of the
exponential of a complex function of a random permutation or subset.
Using this theory,  we find asymptotic expressions for the expected number of copies and induced copies 
of a given graph in a uniformly
random graph with degree sequence $(d_1,\ldots,d_n)$ as $n \rightarrow 
\infty$. 
We also determine the expected number of spanning trees in this model.
The range of degrees covered includes $d_j = \lambda n + O(n^{1/2+\eps})$ 
for some $\lambda$ bounded away from~$0$ and~$1$.
\end{abstract}

\nicebreak

\section{Introduction }\label{S:intro}

For infinitely many natural numbers~$n$, consider 
 vectors $\dvec(n) = (d_1(n),\ldots, d_n(n)) \in \{0,\ldots, n-1\}^n$.
Since we consider asymptotics with respect to $n\to\infty$, we
will generally assume that $n$ is sufficiently large and
just write $\dvec$ in place of $\dvec(n)$, and similarly for other variables.
Everywhere in the paper we assume that 
\[
	\dvec(n) \text{ is a graphical degree sequence;}
\]
that is, there exists a graph  on the vertex set $\{1,\ldots, n\}$  such that 
  $d_j(n)$ is  the degree of vertex $j$, for $j=1,\ldots, n$.
Let $\Gd$ denote the uniform random graph
model of simple graphs on the vertex set $\{1,\ldots, n\}$
with degree sequence~$\dvec$.
By $G\sim\Gd$ we mean that $G$ is a random graph from~$\Gd$.

We study the occurrence of  patterns in  $G\sim\Gd$  such as subgraphs  or  induced  subgraphs  isomorphic to a given graph. Using this theory, we find asymptotic expressions for the expected number of some more general structures,
namely spanning trees and $r$-factors. 
Our aim is to provide formulae that cover sufficiently large and general structures so they could be subsequently used
to estimate moments and derive tail bounds for the limiting distribution of the 
corresponding  random variables.

For any vector $\vvec=(v_1,\ldots, v_t)$, let
\[ \|\vvec\| = \max_{j=1\ldots t} |v_j|\]
denote the infinity norm of $\vvec$.
We also use this norm for functions with finite domain.
We will use the following parameters that depend only on~$\dvec$:
\begin{equation}\label{R-lambda-delta}
  \begin{aligned}
  d &= \frac1n\sum_{j=1}^n d_j,
    & \lambda &= \frac{d}{n-1}, \\
  R &= \frac{1}{n}\sum_{j=1}^n\, (d_j-d)^2,
& \quad
   \delta & = \| (d_1-d,\ldots, d_n-d)\| .
   \end{aligned}
\end{equation}
We  consider  the range of $\dvec$ which satisfy the following assumptions
for some constant $\eta\in(0,\frac12)$ and some constant $\eps>0$
which is sufficiently small depending on $\eta$:
\begin{equation}\label{curlyA}
 \delta \leq n^{1/2 + \eps}\quad\text{and}\quad
 \min\{ \lambda,1-\lambda \} \geq \frac{1}{6\eta^2\log n}. 
\end{equation}
The set of graphs with degrees  $\dvec$ satisfying~\eqref{curlyA} is non-empty for sufficiently large~$n$.  This is implied by the enumeration
results in~\cite{MW90} and also follows directly from the Erd\H{o}s-Gallai characterisation of graphical degree sequences~\cite{EG1960}.
The random graph model~$\Gd$ is thus well-defined.

Let $\calG(n,p)$ denote the binomial model of random graph, in which each edge is present independently with probability~$p$.  
Note that the degree sequence of a random graph from $\calG(n,p)$
satisfies $\delta\leq n^{1/2+\eps}$ with high probability for any
$p=p(n)$.
Our results show that counts  of small subgraphs in $\Gd$ closely match those in
$\calG(n,\lambda)$, but for larger subgraphs the two models diverge and correction factors that we will determine are required.

\medskip
Let~$G$ and~$H$ be graphs with the same vertex set $\{1,2\ldots,n\}$.
The \textit{number of copies of~$H$ in~$G$} is
the number of spanning subgraphs of~$G$ that are isomorphic to~$H$.
For given $\dvec,H$, the random variable $N_{\dvec}(H)$ is the
number of copies of $H$ in  $G$ when  $G$ is  taken at random
from~$\Gd$.
The first problem we consider is the expectation~$\E N_{\dvec}(H)$.
 If $\hvec= (h_1,\ldots, h_n)$
is the degree sequence of $H$ then we define
\begin{equation}
m = \dfrac{1}{2}\,\sum_{j=1}^n h_j, \qquad  \qquad 
\label{Mt} \mu_t = \dfrac{1}{n}\,\sum_{j=1}^n h_j^t  \qquad 
    \text{for $t \geq 1$}. 
\end{equation}
Note that $ m = n\mu_1/2$ is the number of edges of $H$.
Define $\Aut(H)$ to be the automorphism group of $H$,
which is the set of permutations of the vertex set $\{1,\ldots,n\}$ that preserve the edge set of $H$.

\begin{thm}\label{subcount}
For any constant $\eta\in(0,\frac12)$  there is some
$\eps_1(\eta) > 0$ such that the following holds for every fixed 
$\eps \in (0,\eps_1(\eta)]$.
Let $\dvec$ be a degree sequence which satisfies~\eqref{curlyA}.
Suppose that
$H$ is a graph on vertex set $\{ 1,\ldots, n\}$ with $m$
edges and degree sequence $\hvec$
such that 
\begin{equation}
\label{assumption}
m \leq  n^{1+2\eps}, \qquad 
\frac{\delta^3 \mu_3 }{\lambda^3 n^2} \leq n^{-1/2+\eta} \qquad 
\text{and}\qquad
\norm\hvec \leq n^{1/2 + \eps}.
\end{equation}
  Then, as $n \rightarrow \infty$,
  \begin{align*}
       \E N_{\dvec}(H) = \frac{n!}{\card{\Aut(H)}}\, \lambda^{m} \exp\biggl(
    \frac{1-\lambda}{4\lambda}&(\mu_1^2 + 2\mu_1 -2\mu_2)     
    -\frac{R}{2 \lambda^2 n} (\mu_1^2 + \mu_1 - \mu_2) 
    \\  &{}-\frac{1-\lambda^2}{6 \lambda^2 n} \mu_3 
     -\frac{1-\lambda}{\lambda n^2} \sum_{jk\in E(H)} h_j h_k  
          + O(n^{-1/2+\eta}) \biggr).
  \end{align*}
\end{thm}

The result of Theorem 
\ref{subcount} simplifies for graphs $H$ with 
moderate degrees, as shown in the following corollary.

\begin{cor}\label{Cor:subcount} Suppose 
  the assumptions of Theorem \ref{subcount} hold for some  fixed $\eta\in (0,\tfrac12)$ and $\eps\in (0,\eps_1(\eta)]$.
  Suppose also that $\mu_3 \leq\lambda^2 n^{1/2+\eta}$. Then, as $n \rightarrow \infty$,
 \begin{equation}\label{sub-1}
     \E N_{\dvec}(H) = \frac{n!}{\card{\Aut(H)}}\, \lambda^{m} \exp\biggl(
    \frac{1-\lambda}{4\lambda}(\mu_1^2 + 2\mu_1 -2\mu_2)  
      -\frac{R}{2 \lambda^2 n} (\mu_1^2 + \mu_1 - \mu_2)    
     + O(n^{-1/2+\eta})\biggr).
 \end{equation}
 Furthermore, if $\mu_2 \leq n^{-3\eps}$ then
 \[
    \E N_{\dvec}(H) = \frac{n!}{\card{\Aut(H)}} \, \lambda^{m}  
     \(1 + O(n^{-1/2+\eta}+n^{-\eps/2}) \),
 \]
 which matches the binomial random graph model $\calG(n,\lambda)$
  up to the error term.
\end{cor}

McKay~\cite[Theorem~2.8(a,b)]{ranx} gave formulae for the number of
perfect matchings and cycles of given size in $G\sim \Gd$ when
$\dvec = (d,\ldots,d)$ is regular (in other words,  when $\delta =0$).
Applying \eqref{sub-1} in these cases ($H$ is a perfect matching, or a cycle
of a given length) reproduces these expressions when $\dvec$ is
regular, and generalizes them to irregular degree sequences.
Kim, Sudakov and Vu~\cite{KSV2007} obtained a result overlapping the last part of
Corollary~\ref{Cor:subcount} for the case that~$H$ has a constant
number of edges and $\dvec$ is regular with $d=o(n)$.

For regular subgraphs $H$,  we have the following result.

\begin{cor}\label{Cor:regularsubgraphs}
For any constant $\eta\in(0,\frac12)$  and 
 every fixed 
$\eps \in (0,\eps_1(\eta)]$ the following holds  as $n \rightarrow \infty$, where 
$\eps_1(\eta) $ is  provided by  Theorem \ref{subcount}.
Let $\dvec$ be a degree sequence which satisfies~\eqref{curlyA}.
  Suppose also that 
   $h\leq n^{2\eps}$ is a  positive integer and $nh$ is even.
\begin{itemize}\itemsep=0pt
\item[(a)] Let $H$ be an $h$-regular graph. Then
\[ 
  \E N_{\dvec}(H) = 
   \frac{n!}{\card{\Aut(H)}}\, \lambda^{m}
     \exp\biggl( -\frac{1-\lambda}{4\lambda} h(h-2) - \frac{Rh}{2\lambda^2n}
           + O(n^{-1/2+\eta}) \biggr).
\]
\item[(b)] The expected total number
 of $h$-regular spanning subgraphs of $G\sim\Gd$ is
 \[ 
   \frac{\sqrt 2}{(h!)^n} \biggl( \frac{2\lambda m}{e}\biggr)^{\!m}
   \exp\biggl( -\frac{h^2-1}{4} 
      -\frac{1-\lambda}{4\lambda} h(h-2) - \frac{Rh}{2\lambda^2n}
              + O(n^{-1/2+\eta}) \biggr).
 \]
\end{itemize} 
\end{cor}

The proofs of Theorem 
\ref{subcount} and Corollaries \ref{Cor:subcount} and \ref{Cor:regularsubgraphs}
are given in Section \ref{s:subgraph}.

\bigskip

Our second main result concerns the expected number of (labelled) spanning trees in $\Gd$.
This extends, and corrects an error in, a result of McKay~\cite[Theorem~2.8(c)]{ranx}.
McKay considered the regular case only, and gave the first term as $\dfrac{7(1-\lambda)}{2\lambda}$. 
However, the 
correct value is $-\dfrac{1-\lambda}{2\lambda}$, as below.

\begin{thm}\label{treecount}
For any constant $\eta\in(0,\frac12)$  there is some
$\eps_2(\eta) > 0$ such that the following holds
 for every fixed $\eps \in (0,\eps_2(\eta)]$.
Let $\dvec$ be a degree sequence which satisfies~\eqref{curlyA}.
Then, as $n \rightarrow \infty$, the expected number of spanning trees in $G\sim\Gd$ is
  \[
     n^{n-2}\lambda^{n-1} \exp\biggl(
       -\frac{1-\lambda}{2\lambda}
       - \frac{R}{2\lambda^2 n} + O(n^{-1/2+\eta})
    \biggr).
  \]
\end{thm}
The proof of Theorem  \ref{treecount} is given in Section \ref{s:dense}.

\bigskip

 Let~$G$ and $H^{[r]}$ be graphs with vertex sets $\{ 1,\ldots, n\}$ and
 $\{ 1,\ldots, r\}$, respectively.   The
\textit{number of induced copies of~$H^{[r]}$ in~$G$}
is the number
of induced subgraphs of $G$ that are isomorphic to~$H ^{[r]}$.
For given $\dvec,H^{[r]}$, the random variable  $\widetilde{N}_{\dvec}(H^{[r]})$ is the
number of induced copies of $H^{[r]}$ in~$G$ when $G$ is taken
at random from~$\Gd$.
Our third main result estimates the expectation $\widetilde{N}_{\dvec}(H^{[r]})$  when $r$ is not too large.
If  $\hvec^{[r]} = (h_1,\ldots, h_r)$ is the degree sequence of $H^{[r]}$ then we define  
\begin{equation}
\label{omega}
 \omega_{t} = \sum_{j=1}^r (h_j - \lambda(r-1))^t,
  \qquad 
    \text{for $t \geq 1$}.
\end{equation}
Let $m=\frac12\sum_{j=1}^r h_j$ be the number of edges of the graph $H^{[r]}$.
Note that the automorphism group  $\Aut(H^{[r]})$ 
is a subgroup of the group $S_r$ of all
permutations of $\{1,\ldots,r\}$.

\begin{thm}\label{inducedsubcount}
For any constant $\eta\in(0,\frac12)$  there is some
$\eps_3(\eta) > 0$ such that the following holds for every fixed $\eps \in(0,\eps_3(\eta)]$.
Let $\dvec$ be a degree sequence which satisfies~\eqref{curlyA}.
Suppose that $H^{[r]}$ is a graph on vertex set $\{1,\ldots, r\}$ with $m$ edges and
degree sequence $\hvec^{[r]}$ such that 
\begin{equation}
\label{induced-assumption} 
 r\leq n^{1/2 + \eps} \quad  \text{and}\qquad 
  \frac{\delta^3}{\lambda^3(1-\lambda)^3 n^3}
  \sum_{j=1}^r\, \Abs{h_j - \lambda(r-1)}^3 \leq n^{-1/2+\eta}.
\end{equation}
   Then, as $n \rightarrow \infty$,
  \[
       \E \widetilde{N}_{\dvec}(H^{[r]}) = \frac{r!}{|\Aut(H^{[r]})|}\, \binom{n}{r}\,  \lambda^m\, (1-\lambda)^{\binom{r}{2}-m}
       \exp\( \varLambda_0+\varLambda_1+\varLambda_2 + O(n^{-1/2+\eta})\),
  \]
  where
 \begin{align*}
  \varLambda_0 &=
    -\frac{\omega_{2}}{2\lambda(1-\lambda) n} +  
 \frac{R\, \omega_{2}}{2\lambda^2(1-\lambda)^2 n^2}, \\
 \varLambda_1 &=
   \frac{r^2}{2n}+ \frac{(1-2\lambda) \omega_{1}}{2\lambda(1-\lambda) n}
  - \frac{\omega_{1}^2}{4\lambda(1-\lambda) n^2} \\[-0.5ex]
  &{\kern 5em}- \frac{r^2 R}{2\lambda(1-\lambda)n^2}
   - \frac{r\,\omega_{2}}{2\lambda(1-\lambda) n^2} 
  - \frac{(1-2\lambda) \omega_{3}}{6\lambda^2 (1-\lambda)^2 n^2}
  = O\(n^{4\eps}(\log n)^2\), \\
 \varLambda_2 &= - \frac{(1-2\lambda) R\,\omega_{1}}{2\lambda^2(1-\lambda)^2 n^2}
-\frac{r\, \omega_{1} \sum_{j=1}^n (d_j-d)^3}{2\lambda^2(1-\lambda)^2 n^4}
 = O(n^{-1/3+\eta/3+4\eps}).
  \end{align*}
\end{thm}

For induced subgraphs of more moderate order, the terms
$\varLambda_1$ and $\varLambda_2$ fit into the $O(n^{-1/2+\eta})$ error term. 

\begin{cor}
\label{induced-corollary}
 Suppose 
  the assumptions of Theorem \ref{inducedsubcount} hold for some  fixed $\eta\in (0,\tfrac12)$ and $\eps\in (0,\eps_3(\eta)]$.
  Suppose also that
  \[
      r^2(1+\delta^2/n) \leq \lambda^2(1-\lambda)^2 n^{1/2+\eta}.
  \]
Then, as $n \rightarrow \infty$,
\begin{equation}\label{induced-1}  
\E \widetilde{N}_{\dvec}(H^{[r]})  = \frac{r!}{|\Aut(H^{[r]})|}\, \binom{n}{r}\, \lambda^m\, (1-\lambda)^{\binom{r}{2}-m}
  \exp\( \varLambda_0 + O(n^{-1/2+\eta})\).
\end{equation}
Furthermore, if $r \leq n^{1/3-\eps}$ then
\[
  \E \widetilde{N}_{\dvec}(H^{[r]})  =\frac{r!}{|\Aut(H^{[r]})|}\, \binom{n}{r}\, \lambda^m\, (1-\lambda)^{\binom{r}{2}-m}\, \(1 + O(n^{-1/2+\eta}+n^{-\eps/2})\),
\]
which matches the binomial random graph model $\calG(n,\lambda)$ up to the error term.
\end{cor}

Note that assumption~\eqref{induced-assumption} is always satisfied
if $r\leq n^{1/3-\eps}$ and $\eta\geq\frac13$. The proofs of
Theorem~\ref{inducedsubcount} and 
Corollary~\ref{induced-corollary} are given in Section~\ref{s:induced}.

Xiao, Yan, Wu and Ren~\cite{XYWR2008} obtained a result overlapping the last part of
Corollary~\ref{induced-corollary} for the case that~$H$ has
constant size and regular $\dvec = (d,\ldots, d)$ with~$d=o(n)$.
The relationship between the two random graph models $\Gd$ and $\calG(n,\lambda)$  was also studied by 
Krivelevich, Sudakov, Vu and Wormald,
who established concentration near the mean
when $r = O(1)$ and $\dvec=(n/2,\ldots,n/2)$, see \cite[Corollary 2.11]{KSVW}.\
The following includes their result as a special case.

\begin{cor}\label{c:concentration}
For any constant $\eta\in(0,\frac12)$  and 
 every fixed 
$\eps \in (0,\eps_3(\eta)]$ the following holds, where 
$\eps_3(\eta) $ is  provided by  Theorem \ref{inducedsubcount}.
Define $\lambdamin=\min\{\lambda,1-\lambda\}$.
Suppose also that
\[   r \leq (2-\eps)\frac{\log n}{\log \lambdamin^{-1}}. \]
Then, $\E  \widetilde{N}_{\dvec}(H^{[r]})  \to\infty$   as $n \rightarrow \infty$, and 
\[
 \Prob\biggl(\,\biggl| \frac{\widetilde{N}_{\dvec}(H^{[r]})}  {\E \widetilde{N}_{\dvec}(H^{[r]})}  -1\biggr| 
    \geq n^{-\eps/6}+n^{-1/6+\eta/3}\biggr)  = O(n^{-\eps/6}+n^{-1/6+\eta/3}).
\]
\end{cor}

Since a clique is a subgraph if and only if it is an induced subgraph, 
we can use either Theorem \ref{subcount} or Theorem \ref{inducedsubcount} to estimate the expected number of $r$-cliques. 
Taking $H$ to be $K_r$ plus $n-r$ isolated vertices in Theorem~\ref{subcount},
 or $H^{[r]}=K_r$ in Theorem~\ref{inducedsubcount}, we obtain the following corollary.

\begin{thm}\label{Cor_cliques}
  For any constant $\eta\in(0,\frac12)$  there is some
$\eps_4(\eta) > 0$ such that the following holds for every fixed $\eps \in (0,\eps_4(\eta)]$.
Let $\dvec$ be a degree sequence which satisfies~\eqref{curlyA}.
 Then, as $n\rightarrow\infty$, for any positive integer $r$ such that
$r \leq n^{1/2+\eps}$ and 
$\delta^3 r^4/(\lambda^3 n^3)\leq n^{-1/2+\eta}$,
 the expected number of $r$-cliques in $G\in\Gd$  is
\[ 
   \binom{n}{r} \lambda^{\binom{r}{2}}
   \exp\left( - \frac{(1-\lambda) r^2(r-3)}{2\lambda n}+
    \frac{Rr^3}{ 2 \lambda^2 n^2} 
    -\frac{(1-\lambda)(2+5\lambda)r^4}{12 \lambda^2n^2}
    +  O(n^{-1/2+\eta})\right).
\]
\end{thm}
The formula for the number of independent subsets of size $r$
can be obtained from the formula given in Corollary \ref{Cor_cliques} 
by simply swapping the roles of $\lambda$ and $1-\lambda$.

\subsection{Outline of our approach}\label{s:approach}

Our proofs are based on the asymptotic enumeration results of McKay~\cite{ranx}.
To illustrate the nature of our task, the proof of Theorem~\ref{subcount}
relies on a theorem from~\cite{ranx}, here quoted as Theorem~\ref{treeprob},
that the probability of a subgraph $H$ appearing in a fixed location 
in $G\sim \calG_{\dvec}$ has the form
\begin{equation}\label{explain}
    \Prob(\dvec,H) = \lambda^m e^{F(\dvec,H) + o(1)}
\end{equation}
for a certain function~$F$.
In order to find the expectation of the number of all appearances of isomorphs
of $H$ as subgraphs, we need to sum $\Prob(\dvec,H')$ over all $H'\cong H$.
Clearly, this is equivalent to finding the expectation of
$e^{F(\dvec^\sigma,H) +o(1)}$, where $\dvec^\sigma$ is a uniformly
random permutation~$\sigma $
of the entries of~$\dvec$.

Since the function $F(\dvec,H)$ in~\eqref{explain} is too large to allow
useful expansion of the exponential, we must estimate $\E e^{F(\dvec^\sigma,H)}$
directly.
We do this by applying the theory of exponentials of martingales
developed in~\cite{mother}, which we summarize in Section~\ref{s:sums}.
In order to facilitate similar applications in the future, 
in Section~\ref{s:randomperm} we prove some general theorems about the
expectations of the exponentials of functions of random permutations.
Theorem~\ref{subcount} and its corollaries are  proved in 
Section~\ref{s:subgraph}.
The proof of Theorem \ref{inducedsubcount} is  given in 
Section \ref{s:induced}.  It follows by a similar argument starting from \cite[Theorem 2.4]{ranx}, which is quoted here as Theorem \ref{induced-prob}.

The first $k$ entries in a random permutation form a
random $k$-subset, so the same theorems can be used to estimate the expectations
of the exponentials of functions of random subsets, and thereby functions
of hypergeometric and multinomial distributions.  
We use this theory to prove Theorem \ref{treecount} in Section~\ref{s:dense},
as multinomial distributions appear naturally for counts of trees with given degrees.

\nicebreak
\section{Expectations of exponentials}\label{s:sums}

First, in Section~\ref{s:results} we review some notation and results from~\cite{mother}.
Then, in Section~\ref{s:randomperm} and~\ref{s:discrete},
 we prove some auxiliary results which will help us to apply
the machinery from~\cite{mother} in the discrete setting.

In this paper, we will only apply the machinery of this section to 
real-valued martingales.  However, 
the complex-valued discrete setting is also covered in this section, 
in order to provide bounds which may be useful for future applications. 
In particular, such bounds can be useful for determining asymptotic distributions
by analysis of the corresponding characteristic functions  (Fourier inversion).

Given a complex random variable $Z$, two types of squared variation
are commonly defined.
The variance is
\begin{align*}
\Var Z &= \E\,\abs{Z-\E Z}^2 = \E\,\abs{Z}^2 - \abs{\E Z}^2
=\Var\Re Z+\Var\Im Z, \\
\intertext{while the pseudovariance is}
\V Z &= \E\,(Z-\E Z)^2 = \E Z^2 - (\E Z)^2
= \Var\Re Z - \Var\Im Z + 2i\myCov{\Re Z,\Im Z}.
\end{align*}
We will need both. Of course, they are equal for real random variables.

\subsection{Complex martingales}\label{s:results}

Let $\P=(\varOmega,\calF,\Prob)$ be a probability space.
A sequence $\F = \calF_0,\ldots,\calF_n$ of $\sigma$-subfields
of $\calF$ is a \textit{filter} if $\calF_0\subseteq\cdots\subseteq\calF_n$.
A sequence $Z_0,\ldots,Z_n$ of random variables on $\P=(\varOmega,\calF,\Prob)$
is a \textit{martingale with respect to $\F$} if
\begin{itemize}\itemsep=0pt
\item[(i)] $Z_j$ is $\calF_j$-measurable and has finite expectation, for $j=0,\ldots, n$;
\item[(ii)] $\myE{Z_j \st \calF_{j-1}} = Z_{j-1}$ for $j=1,\ldots, n$.
\end{itemize}
Observe that $Z_j=\myE{Z_n \st \calF_j}$ a.s.\ for each~$j=0,\ldots, n$.

Let $Z$ be a random variable on $\P$. We use the following notation
for statistics conditional on $\calF_j$, for $j=0,\ldots, n$:
\begin{align*}
  \E_j Z &= \myE{Z\st\calF_j}, \\
  \V_j  Z &= \myE{(Z-\E_jZ)^2\st\calF_j}
           = \E_j Z^2 - (\E_j Z)^2, \\
  \Diam_j Z &= \myDiam{Z\st\calF_j}.
\end{align*}
Here the \textit{conditional diameter} of $Z$ with respect to $\sigma$-subfield $\calF'$ of $\calF$ is defined as 
\begin{equation}\label{diam3}
   \myDiam {Z\st\calF'} = \sup_{\theta\in (-\pi,\pi]}\, 
     \Bigl[\mYesssup{\Re(e^{-i\theta}Z)\st\calF'}
         + \mYesssup{-\Re(e^{-i\theta}Z)\st\calF'} \Bigr],
\end{equation}
where the conditional essential supremum of a real random
variable $X$ with $\abs X\leq c$ a.s.~can be defined (see \cite{Barron}) by
\[ 
  \myesssup {X\st\calF'} = -c + \lim_{r\to\infty} \(\myE{(X+c)^r\st\calF'}\)^{1/r}.
\]
When $Z$ is real, we can restrict~\eqref{diam3} to $\theta=0$
and then $\myDiam{Z\st\calF'}$ is the same as the conditional
range defined by McDiarmid~\cite{McDiarmid}. 
In the trivial case $\calF' = \{\emptyset, \varOmega\}$, the 
(unconditional) diameter can be alternatively  defined by
\begin{equation}\label{diam}
   \Diam Z = \myDiam{Z\st\calF'} = \esssup\,\abs{Z-Z'}, 
   \quad\text{where $Z'$ is an independent copy of~$Z$}.
\end{equation}
For more information about conditional essential supremum and conditional diameter, 
see, for example, \cite{Barron} and \cite[Section 2.1]{mother}. 
We will use the fact that the diameter and conditional diameter are seminorms and
so, in particular, they are subadditive.

The following first-order and second-order estimates were proved
in~\cite[Theorem 2.7 and Theorem 2.9]{mother}, and are stated below
for convenience.

\begin{thm}\label{goodthm}
	Let $\Z = Z_0,Z_1,\ldots, Z_n$ be an a.s.~bounded complex-valued 
	martingale with respect to a filter $\calF_0,\ldots, \calF_n$.
	For $j=1,\ldots, n$, define
	\begin{subequations}\label{T2cond}
		\begin{align*}
			R_j = \Diam_{j-1} Z_j, \ \ \ \
                         Q_j = \max \bigl\{\Diam_{j-1}\E_j (Z_n - Z_{j})^2 ,
                		 \Diam_{j-1} \E_j(\Re Z_n - \Re Z_{j})^2 \,\bigr\}.
		\end{align*}
	\end{subequations}
	Then the following estimates hold. 
\begin{itemize}\itemsep=0pt
 \item[(a)]   $\E_0 e^{Z_n}   = e^{Z_0}(1 + K(\Z))$,
where $K(\Z)$ is an $\calF_0$-measurable random variable with
\[
  \abs{K(\Z)}  \leq \mYesssup{ e^{\frac18\sum_{j=1}^n R_j^2} \St \calF_0}
      - 1 \ \text{ a.s.}
\]
\item[(b)] 
	$
		\E_0 e^{Z_n} = e^{Z_0 + \frac{1}{2} \V_0 Z_n} \(1 + L(\Z) e^{\frac{1}{2} \V_0[\Im Z_n]}\), 
	$
		where $L(\Z)$  is an 
	$\calF_0$-measurable random variable with 
	\[
		\abs{L(\Z)} \leq
	\MYesssup{
		\exp \Bigl( \,
\sum_{j=1}^n \(\tfrac{1}{6} R_j^3 + \tfrac{1}{6}R_jQ_j + \tfrac{5}{8} R_j^4 + \tfrac{5}{32} Q_j^2\)
		\Bigr) \biggm| \calF_0} -1 \ \text{ a.s.}
	\]
	\end{itemize}
\end{thm}

The following lemma, proved in~\cite[Lemma 2.8]{mother}, is useful for bounding the quantities~$Q_j$ when
applying Theorem~\ref{goodthm}(b).

\begin{lemma}\label{telescope}
Under the conditions of Theorem~\ref{goodthm}, we have
\begin{align*}
   \E_j (Z_n-Z_j)^2 =
   \sum_{k=j+1}^n \E_j(Z_k-Z_{k-1})^2
\end{align*}
for $0\leq j\leq n$.
\end{lemma}

An important example of a martingale is made by the Doob martingale process. Suppose 
$\X=(X_1,\ldots,X_n)$ is a random vector on $\P$ and $f(\X)$ is 
a complex random variable of bounded expectation. Consider the filter $\calF_0,\ldots,\calF_n$ defined by
$\calF_j = \sigma(X_1,\ldots,X_j)$, where $\sigma(X_1,\ldots,X_j)$ denotes the $\sigma$-field 
generated by the random variables $X_1,\ldots, X_j$. In particular, $\calF_0 =\{\emptyset, \varOmega\}$ and $\E_0$ is
the ordinary expectation.
Then we have the martingale 
\[\hspace*{3cm} Z_j = \myE {f(X_1,\ldots,X_n) \st \calF_j}, \hspace*{3cm} j=0,\ldots,n. \]
It was shown in \cite[Lemma 3.1]{mother} that for this case 
the conditional diameter satisfies the following property:
\begin{equation}
\begin{aligned}
\text{
$\Diam_j f(\X)$ has the same distribution as 
$\delta_j (X_1,\ldots,X_j)$, where}\\
	\delta_j (x_1,\ldots,x_j)      = 
	\mYDiam{f(x_1,x_2,\ldots, x_j, X_{j+1}, \ldots, X_n)}.
\end{aligned}
\label{henry}
\end{equation}
Here the variables $X_j, \ldots, X_n$ are random and $x_1,\ldots,x_j$ are fixed.

\subsection{Random permutations}\label{s:randomperm}

Let $S_n$ denote the set of permutations of $\{1,\ldots, n\}$. 
We will write a permutation as a vector: if $\omega\in S_n$ maps $j$ to $\omega_j$
for $j=1,\ldots, n$ then we write $\omega = (\omega_1,\omega_2,\ldots, \omega_n)$.
For any $\omega,\sigma\in S_n$, define
\[ \omega\circ \sigma = (\omega_{\sigma_1}, \ldots, \omega_{\sigma_n}).\]
That is, $\sigma$ acts on $\omega$ on the right by permuting the \emph{positions}
of $\omega$, not the values.  

Now suppose $\X=(X_1,\ldots,X_n)$ is a uniformly random 
element of $S_n$. 
Although the random variables $X_1, \ldots, X_n$ are dependent,  the Doob martingale process is still applicable:
for a given permutation $\omega=(\omega_1,\ldots \omega_n)  \in S_n$  and the function $f: S_n \to \Complexes$  define
\begin{equation}\label{Doob_permutations}
	Z_k(\omega) = \myE{f(\X) \st X_j=\omega_j, \ 1\leq j\leq k}.
\end{equation}
The sequence $Z_0(\X), Z_1(\X), \ldots, Z_n(\X)$ is a martingale
with respect to the filter $\calF_0, \allowbreak\ldots, \calF_n$,
where for each $k$, the $\sigma$-field
$\calF_k$
is generated by the sets 
\[
  \varOmega_{k,\sigma} = \{\omega \in S_n \st \omega_j=\sigma_j, \ 1\leq j\leq k\}
\]
for all  
$k$-tuples $(\sigma_1,\ldots, \sigma_k)$ with distinct components.
From now on we simply write $Z_k$ instead of $Z_k(\X)$, for $k=0,\ldots, n$.

Since $Z_n=Z_{n-1}$ and $\calF_n=\calF_{n-1}$, we will find it convenient
to stop the martingale at $Z_{n-1}$.	
In the following we will  use the notations of Section~\ref{s:results} for statistics
	conditional on~$\calF_k$. 

Given a function $f:S_n\to \Complexes$, we  use the infinity norm
\[ \| f\| = \max_{\omega\in S_n} | f(\omega)|.\]
For any $j,a\in \{1,\ldots, n\}$, and any $\omega\in S_n$, define
\[
  D^{(j\, a)} f(\omega) = f(\omega) - f(\omega\circ(j\, a)). 
\]
Here $(j\, a)\in S_n$ is the transposition which exchanges  $j$ and $a$.
Now, let
\begin{align*}
 \hspace*{8mm} \alpha_j[f,S_n] &= \frac{1}{n-j} \sum_{a=j+1}^n \|  D^{(j\, a)} f\|,
           \hspace*{37mm} \qquad (1\leq j\leq n-1) \\
 \hspace*{8mm} \varDelta_{jk}[f,S_n] &= \frac{1}{(n-j)(n-k)}
                             \sum_{a=j+1}^n\,\sum_{b=k+1}^n 
  \| D^{(k\, b)} D^{(j\, a)} f\|,
           \hspace*{8mm}  (1\leq j\neq k \leq n-1).
\end{align*}
Note that the parameters $\alpha_j$ and $\varDelta_{jk}$ satisfy the triangle inequality:
\begin{equation}\label{triangle_aD}
  \begin{aligned}
\alpha_j[f+f',S_n] &\leq \alpha_j[f,S_n] + \alpha_j[f',S_n], \\ \varDelta_{jk}[f+f',S_n] &\leq \varDelta_{jk}[f,S_n] +
\varDelta_{jk}[f',S_n].
   \end{aligned}
\end{equation}

The following lemma provides bounds on the quantities that arise in Theorem~\ref{goodthm}. 

\begin{lemma}
\label{Lemma_Perm}
Let $\X=(X_1,\ldots,X_n)$ be a uniformly random element of $S_n$. 
Let $f: S_n \rightarrow \Complexes$ and let
$(Z_0, Z_1, \ldots, Z_{n-1})$ be the Doob martingale sequence given by 
\emph{\eqref{Doob_permutations}}.
Write $\alpha_k = \alpha_k[f,S_n]$ and $\varDelta_{jk} = \varDelta_{jk}[f,S_n]$. Then
\begin{align}
	\Diam_{j-1} Z_j &\leq \alpha_j,  & 1 \leq j \leq n-1, \label{Perm_alpha}\\ 
	\Diam_{j-1} \E_j (Z_k-Z_{k-1})^2 &\leq 2 \alpha_k \varDelta_{jk}, & 1 \leq j<k \leq n-1. 
	\label{Perm_Delta}
\end{align}
\end{lemma}
\begin{proof}
Firstly, observe that $Z_j$ can be represented by a function of $j$ arguments:   
\[ 
 Z_j(\omega) = f_j(\omega_1,\ldots, \omega_j), \ \ \ \text{  } \omega \in S_n.
\]
Recalling (\ref{diam}) and (\ref{henry}), we have 
\[
	\Diam_{j-1} Z_j   = \max\, \Abs{f_j(\sigma_1,\ldots,\sigma_j) - f_j(\sigma_1,\ldots,\sigma_{j-1}, \sigma_j')},
\]
where the maximum is taken over all $j$-tuples $(\sigma_1,\ldots,\sigma_j)$ 
with distinct components, and $\sigma_j' \neq \sigma_1,\ldots, \sigma_{j-1}$. 
By definition of $f_j$ and $Z_j$, we have
\begin{align*}
 &\Abs{f_j(\sigma_1,\ldots, \sigma_j)   - f_j(\sigma_1,\ldots, \sigma_{j-1},\sigma'_j)} \notag\\
  &= \Abs{ \myE{f(\X) \mid  X_1 = \sigma_1,\ldots, X_j = \sigma_j} - 
         \myE{f(\X)\mid X_1 = \sigma_1,\ldots, X_{j-1} = \sigma_{j-1},\, X_j = \sigma'_j}}\notag\\
 &= \biggl|\, \frac{1}{n-j} \sum_{a=j+1}^n \mYE{ D^{(j\, a)}f(\X) \St
   X_1 = \sigma_1,\ldots, X_j = \sigma_j,\,\, X_a = \sigma'_j} \,\biggr|,
\end{align*}
since $\sigma'_j$ must occupy some position $a\in \{ j+1,\ldots, n\}$ in $\sigma$, and by symmetry each
possibility is equally likely.  Therefore
\begin{equation}
 |f_j(\sigma_1,\ldots, \sigma_j)   - f_j(\sigma_1,\ldots, \sigma_{j-1},\sigma'_j)| 
  \leq \frac{1}{n-j}\, \sum_{a=j+1}^n \| D^{(j\, a)} f\| 
  =\alpha_j,  \label{double-dagger}
\end{equation}
which implies the bound \eqref{Perm_alpha} for  $\Diam_{j-1} Z_j$. 

Now we proceed to the bound for $\Diam_{j-1} \E_j (Z_k-Z_{k-1})^2$.
Define  $\Tilde{f}: S_n \rightarrow \Complexes$ by
\[
   \Tilde{f}(\omega) = (Z_k(\omega) - Z_{k-1}(\omega))^2  
   = (f_k(\omega_1,\ldots,\omega_k)-f_{k-1}(\omega_1,\ldots,\omega_{k-1}))^2.
\]
Since $D^{(j\, a)} \Tilde{f}(\omega)$ is the difference of 
two squares, we have
\begin{align*}
 D^{(j\, a)} \Tilde{f}(\omega) 
      &= \Tilde{f}(\omega) - \Tilde{f}(\omega\circ (j\, a))\\
 &= \(Z_k(\omega) - Z_{k-1}(\omega) + Z_k(\omega\circ (j\, a)) -
     Z_{k-1}(\omega\circ (j\, a))\)\\
 & {\qquad} 
 \times \(Z_k(\omega) - Z_{k-1}(\omega) - Z_k(\omega\circ (j\, a))
     + Z_{k-1}(\omega\circ (j\, a))\).
\end{align*}
Using \eqref{Perm_alpha} applied to $f$, we have
$\Abs{Z_k(\omega) - Z_{k-1}(\omega)} \leq \Diam_{k-1} Z_k$
and hence
\[
 \Abs{Z_k(\omega) - Z_{k-1}(\omega) + Z_k(\omega\circ (j\, a))
        - Z_{k-1}(\omega\circ (j\, a))} \leq 2\,\Diam_{k-1} Z_k \leq 2\alpha_k.
\]
Therefore, applying \eqref{Perm_alpha} to $\Tilde{f}$ gives
\begin{align}
	 & \Diam_{j-1} \E_k(Z_k- Z_{k-1})^2  \notag\\
        &{\qquad} \leq \alpha_j[\Tilde{f}, S_n]
	 =  \frac{1}{n-j}
	\sum\limits_{a=j+1}\limits^n  \| D^{(j\, a)} \Tilde{f}\, \|  \notag\\
	 &{\qquad} \leq \frac{2\alpha_k }{n-j}\,
	 \sum\limits_{a=j+1}\limits^n \max\limits_{\omega \in S_n } \,
	 \Abs{ Z_k(\omega)- Z_{k-1}(\omega) - Z_k(\omega{\circ(j\, a)})
	    + Z_{k-1}(\omega{\circ(j\, a)})}.
\label{reduced}
\end{align}
In the remainder of the proof we work towards an upper bound on the summand.

For any $c\in \{1,\ldots, n\}$
and any permutation $(k\, b)$, with $1\leq k \leq b\leq n$
(either a transposition or the identity permutation),
write $c^{(k\, b)}$ for the image of $c$ under the action of $(k\, b)$.
Given $k\in \{1,\ldots, n\}$, define the set
\[ I_k = \{ (b,c)\mid k\leq b\neq c\leq n\}\]
of distinct ordered pairs with both entries at least $k$.

Now we consider two cases.  
\bigskip

\noindent {\bf Case 1.}\ Firstly, suppose that $a\in \{ k+1,\ldots, n\}$.
To begin, observe that 
\begin{align}
  Z_{k-1}(\omega) &= \frac{1}{(n-k)(n-k+1)} \notag \\
  & \quad {} \times  \sum_{(b,c)\in I_k} 
      \mYE{f(\X)\St X_1 = \omega_1,\ldots, X_{k-1} = \omega_{k-1},
              X_b = \omega_k,\, X_c = \omega_a} \label{case1a}
\end{align}
using arguments similar to those which led to (\ref{double-dagger}). 
Next, let
$\Tilde{\X} = \X\circ (k\, b)$, which is also a uniformly random element
of $S_n$, and write
\begin{align*}
 \mYE{f(\X\circ (k\, b))\St X_1 = &\omega_1,\ldots, X_{k-1}=\omega_{k-1},\,
   X_b = \omega_k,\, X_c = \omega_a} \\
 &= \mYE{f(\Tilde{\X})\St \Tilde{X}_1 = \omega_1,\ldots, \Tilde{X}_k =
   \omega_k,\, \Tilde{X}_{c^{(k\, b)}} = \omega_a}.
\end{align*}
Note that $c' = c^{(k \, b)}$ ranges over $\{ k+1,\ldots, n\}$ as $c$ ranges
over $\{ k,\ldots, n\}\setminus \{ b\}$. 
Therefore
\begin{align}
&\frac{1}{(n-k)(n-k+1)}
   \sum_{(b,c)\in I_k}\! \mYE{f(\X\circ (k\, b))\St X_1 = \omega_1,\ldots, 
   X_{k-1}=\omega_{k-1},\, X_b = \omega_k,\, X_c = \omega_a} \notag\\
 &\quad = \frac{1}{(n-k)(n-k+1)}\, \sum_{b=k}^n \sum_{c' = k+1}^n \, 
  \mYE{f(\Tilde{\X})\mid \Tilde{X}_1 = \omega_1,\ldots, \Tilde{X}_k =
   \omega_k,\, \Tilde{X}_{c'} = \omega_a} \label{midway}\\
 &\quad = \frac{1}{n-k}\, \sum_{c'=k+1}^n \mYE{f(\Tilde{X})\St \Tilde{X}_1 = \omega_1,
  \ldots, \Tilde{X}_k = \omega_k,\, \Tilde{X}_{c'} = \omega_a} \notag\\
  &\quad = \mYE{f(\Tilde{\X})\St \Tilde{X}_1 = \omega_1,\ldots,
   \Tilde{X}_k = \omega_k} = Z_k(\omega),\label{case1b}
\end{align}
similarly to (\ref{double-dagger}), since the summand in (\ref{midway})
is independent of $b$.

Arguing as above with $\Tilde{\X} = \X \circ (j \, c)$ gives
\begin{align}
 Z_{k-1}&(\omega\circ (j\, a)) \notag \\
  &= f_{k-1}(\omega_1,\ldots, \omega_{j-1}, \omega_a, \omega_{j+1},\ldots, \omega_{k-1})\notag \\
  &= \frac{1}{(n-k)(n-k+1)} \notag \\
  & \quad {~}\times \sum_{(b,c)\in I_k}
  \mYE{ f(\X\circ (j\, c))\St X_1 = \omega_1,\ldots, X_{k-1} = \omega_{k-1},\,
   X_b = \omega_k,\, X_c = \omega_a}.
 \label{case1c}
\end{align}

Finally, let $\Tilde{\X} = \X \circ (j\, c)\circ (k\, b)$, which is a uniformly
random element of $S_n$, and recall that $c^{(k\, b)}$ ranges over $\{ k+1,\ldots, n\}$
as $c$ runs over $\{ k,\ldots, n\} \setminus \{ b\}$.  Arguing as above gives 
\begin{align}
  & \frac{1}{(n-k)(n-k+1)} \notag \\
  & {~~} \times \sum_{(b,c)\in I_k}
  \mYE{ f(\X\circ (j\, c)\circ (k\, b))\St X_1 = \omega_1,\ldots, X_{k-1} = \omega_{k-1},\,
   X_b = \omega_k,\, X_c = \omega_a}\notag \\
 &= \frac{1}{(n-k)(n-k+1)}\notag \\
 & {~~}\times \sum_{(b,c)\in I_k} \E\bigl[ f(\Tilde{\X}) \St \Tilde{X}_1 = \omega_1,
  \ldots, \Tilde{X}_{j-1} = \omega_{j-1},\, \Tilde{X}_j = \omega_a,
  \, \Tilde{X}_{j+1} = \omega_{j+1},\ldots,
 \notag \\[-1.5ex]
  &  \hspace*{19em}
   \Tilde{X}_{k-1} = \omega_{k-1}, \Tilde{X}_k = \omega_k,\,
   \Tilde{X}_{c^{(k\, b)}} = \omega_j\bigr]\notag \\
 &= \frac{1}{n-k+1}\, \sum_{b=k}^n Z_k(\omega\circ (j\, a))
 = Z_k(\omega\circ (j\, a)).
\label{case1d}
\end{align}

\medskip
\noindent Combining (\ref{case1a})--(\ref{case1d}) together, we find that when $a\in \{ k+1,\ldots, n\}$,
\begin{align}
\Abs{Z_k&(\omega) - Z_{k-1}(\omega) - Z_k(\omega\circ (j\, a))  +
    Z_{k-1}(\omega\circ (j\, a))} \notag\\
  &= \frac{1}{(n-k)(n-k+1)}\notag \\
  & \quad {} \times \sum_{(b,c)\in I_k} \mYE{ D^{(k\, b)} D^{(j\, a)} f(\X) \St X_1 = \omega_1,\ldots,
          X_{k-1} = \omega_{k-1},\, X_b = \omega_k,\, X_c = \omega_a}\notag\\
  &\leq \frac{1}{(n-k)^2}\sum_{(b,c)\in I_k} \| D^{(k\, b)} D^{(j\, c)} f\|.
\label{case1}
\end{align}

\medskip

\noindent {\bf Case 2.}\
Now suppose that $a\in \{ j+1,\ldots, k\}$.
Define $z=b$ if $a=k$, and $z=a$ if $a\in\{j+1,\ldots,k-1\}$.
Arguing as above, we have
\begin{align*}
Z_k(\omega) &= \frac{1}{n-k+1}\,\sum_{b=k}^n
  \mYE{ f(\X\circ (k\, b))\St 
   X_1 = \omega_1,\ldots, X_{k-1} = \omega_{k-1},\, X_b = \omega_k},\\
Z_{k-1}(\omega) &= \frac{1}{n-k+1}\, \sum_{b=k}^n
  \mYE{ f(\X)\St 
   X_1 = \omega_1,\ldots, X_{k-1} = \omega_{k-1},\, X_b = \omega_k},\\
Z_{k-1}(\omega\circ (j\, a)) &= \frac{1}{n-k+1}\, \sum_{b=k}^n
  \mYE{ f(\X\circ (j\, z))\St 
   X_1 = \omega_1,\ldots, X_{k-1} = \omega_{k-1},\, X_b = \omega_k},\\
Z_{k}(\omega\circ (j\, a)) &= \frac{1}{n-k+1}\\
 &{\quad}\times \sum_{b=k}^n    
  \mYE{ f(\X\circ (j\, z)\circ (k\, b))\St 
   X_1 = \omega_1,\ldots, X_{k-1} = \omega_{k-1},\, X_b = \omega_k}.
\end{align*}
Combining these, we find that when $a\in \{ j+1,\ldots, k\}$,
\begin{equation}
\label{case2}
 \Abs{Z_k(\omega) - Z_{k-1}(\omega) - Z_k(\omega\circ (j\, a)) +
    Z_{k-1}(\omega\circ (j\, a))}
  \leq \frac{1}{n-k}\,\sum_{b=k}^n \| D^{(k\, b)} D^{(j\, z)} f\|.
\end{equation}

\noindent {\bf Consolidation.}\
Now we perform the sum over $a$. From (\ref{case1}) and (\ref{case2})
we have
\begin{align}
 \sum_{a=j+1}^n \Abs{Z_k(\omega) - Z_{k-1}(\omega) - Z_k(\omega\circ (j\, a))
   &{}+ Z_{k-1}(\omega\circ (j\, a))} \notag\\
\leq   \frac{1}{n-k} \sum_{(b,c)\in I_k}
    \| D^{(k\, b)}\, D^{(j\, c)} f\|
  &{}+ \frac{1}{n-k} \sum_{b=k}^n \| D^{(k\, b)} D^{(j\, b)} f\|  \notag\\
   &{}+ \frac{1}{n-k} \sum_{a=j+1}^{k-1} \sum_{b=k}^n \| D^{(k\, b)} D^{(j\, a)} f\|,
   \label{entiresum}
\end{align}
using the fact that (\ref{case1}) is independent of $a$ in Case 1.
Replacing the dummy variable $c$ in the first sum by $a$, and observing
that any term with $k=b$ equals zero, we can rewrite the right-hand side of
\eqref{entiresum} as
\[ \frac{1}{n-k}\sum_{a=j+1}^n \sum_{b=k}^n \| D^{(k\, b)} D^{(j\, a)} f\|
 =  \frac{1}{n-k}\sum_{a=j+1}^n \sum_{b=k+1}^n \| D^{(k\, b)} D^{(j\, a)} f\|
  = (n-j)\, \varDelta_{jk}.
\]
Substituting this into (\ref{reduced}), we conclude that
\[
    \Diam_{j-1} \E_j (Z_k-Z_{k-1})^2 \leq 2\alpha_k\,\varDelta_{jk}
\]
as required.
\end{proof}

Combining the bounds proved above with Theorem~\ref{goodthm}
gives the following.

\begin{thm}\label{RanPermThm}
Let $\X$ be a uniformly random element of $S_n$
and let $f: S_n \rightarrow \Complexes$.
Write $\alpha_k = \alpha_k(f,S_n)$ and
$\varDelta_{jk} = \varDelta_{jk}(f,S_n)$. Then
\begin{itemize}\itemsep=0pt
 \item[(a)]   $\E e^{f(\X)}   = e^{\E f(\X)}(1 + K(f))$,
where $K(f)\in\Complexes$ satisfies
\[
  \abs{K(f)}  \leq e^{\frac18\sum_{j=1}^{n-1} \alpha_j^2} - 1.
\]
\item[(b)] 
$
\E e^{f(\X)} = e^{\E f(\X) + \frac{1}{2} \V f(\X)} \(1 + L(f) e^{\frac{1}{2} \Var\Im f(\X)}\), 
$
where $\beta_j= \sum_{k=j+1}^{n-1}\alpha_k\varDelta_{jk}$ and\\ $L(f)\in\Complexes$ satisfies 
\[
		\abs{L(f)} \leq
		\exp \biggl( \,
\sum_{j=1}^{n-1} \,\(\tfrac{1}{6} \alpha_j^3 
+ \tfrac{1}{3}\alpha_j \beta_j + \tfrac{5}{8} \alpha_j^4 + \tfrac{5}{8}
\beta_j^2\)	 \biggr) -1.
\]
\end{itemize}
 
\end{thm}
\begin{proof}
Let $\Z(\X)=(Z_0, Z_1, \ldots, Z_{n-1})$ be the Doob martingale sequence given by (\ref{Doob_permutations}).
By applying Theorem~\ref{goodthm} to $\Z(\X)$,  it remains  to show that 
\[
	   R_j  \leq \alpha_j,  \qquad
	Q_j \leq 2 \beta_j.
 \]
The first bound is given by $\eqref{Perm_alpha}$ and the definition
 of~$R_j$.

Observe that $D^{(ja)} (\Re f( \omega)) = \Re D^{(ja)}  f( \omega) $ 
for any $j,a\in\{1,\ldots,n\}$. Therefore, 
\begin{align*}
\alpha_j[\Re f, S_n] &\leq \alpha_j[f,S_n],
 \hspace*{7em} (1 \leq j \leq n-1);
\\
 \varDelta_{jk}[\Re f, S_n] &\leq   \varDelta_{jk} [f,S_n],
 \hspace*{6.5em} (1 \leq j\neq k \leq n-1).
\end{align*}
Using $\eqref{Perm_Delta}$ twice (for  $f$ and $\Re f$), we find that both quantities 
$\Diam_{j-1} \E_j (Z_k-Z_{k-1})^2$
and $\Diam_{j-1} \E_j (\Re Z_k-\Re Z_{k-1})^2$
are bounded above by $2 \alpha_k\,\varDelta_{jk}$.
Since the conditional diameter is subadditive, we can
apply Lemma~\ref{telescope} to obtain the remaining bound on $Q_j$.
\end{proof}


\subsection{Random subsets and other discrete distributions}\label{s:discrete}

Using our estimates for random permutations, we can also apply Theorem \ref{goodthm} for
functions of random subsets of given size, as well as functions of random
vectors with standard multidimensional discrete distributions, such as the hypergeometric
distribution or the multinomial distribution.
We now define analogues of the operator $D^{(j\, a)}$
for these cases.

\bigskip

\noindent {\bf Subsets.}\
Let $2^{[n]}$ denote the set of all subsets of $\{1,2,\ldots, n\}$.
For a given $f:2^{[n]}\rightarrow \Complexes$, and for every $A\in 2^{[n]}$, let 
\[ D_B^{(j\, a)} f(A) = f(A) - f(A\oplus \{j,\, a\})
\]
where $\oplus$ denotes the symmetric difference.
Note that if $|A\cap \{ j,\,a\}|=1$ then $A\oplus \{j,a\}$ has the same
size as $A$.
Let $B_{n,m}$ denote the set of $m$-subsets of  $\{1,\ldots, n\}$, 
and define
\[
 \alphamax[f,B_{n,m}] = \max \,\Abs{D_B^{(j\, a)} f(A)},
\]
where the maximum is taken over all $A \in B_{n,m}$ and 
all $j, a\in \{1,\ldots, n\}$ such that $j\in A$ and $a\not\in A$.
Similarly, define
\[
	\Deltamax[f,B_{n,m}] = \max\,\Abs{ D_B^{(k\, b)} \, D_B^{(j\, a)} f(A) },
\]
where the maximum is over all distinct $j,k,a,b\in \{1,\ldots, n\}$ and all $A \in B_{n,m}$ such that 
$j,k\in A$ and $a,b\not\in A$.
Note that $\alphamax[f,B_{n,m}]$ and
$\Deltamax[f,B_{n,m}]$ depend only on the values of $f$ on the set $B_{n,m}$.

\bigskip

\noindent {\bf Sequences.}\
For a given function $f:\Integers^{\ell} \to\Complexes$, and for every $\xvec = (x_1,\ldots, x_\ell)\in \Integers^\ell$, define
\[ D_N^{(j\, a)} f(\xvec) = f(\xvec) - f(\xvec')\]
where 
$\xvec'$ has all entries equal to those of $\xvec$, except that the $j$-th entry is
increased by 1 and the $a$-th entry is decreased by 1.
For positive integers $\ell, m$, define
\[
   N_{\ell,m} = \bigl\{ (x_1,\ldots,x_\ell)\in \{0,1,2,\ldots\,\}^\ell : x_1+\cdots+x_\ell=m \bigr\}.
\]
Note that if $\xvec\in N_{\ell,m}$ with $x_a > 0$ then
$\xvec'$, defined above, also belongs to $N_{\ell,m}$.
(If $\xvec$ has any positive entry then no other entry can equal $m$.)
Define
\begin{equation}
\label{alphamax-vector}
 \alphamax[f,N_{\ell,m}] = \max \,\Abs{D_N^{(j\, a)} f(\xvec)}
\end{equation}
 where the maximum is over all $\xvec\in N_{\ell,m}$ and all distinct $j,a$ such that
$x_a > 0$.
Also define 
\begin{equation}
\label{Deltamax-vector}
    \Deltamax[f,N_{\ell,m}] = 
       \max \,\Abs{D_N^{(k\, b)} \, D_N^{(j\, a)} f(\xvec)}
\end{equation}
 where the maximum is taken over all distinct $j,k, a,b$,
such that $\min\{ x_a, x_b\} > 0$ and all $\xvec\in N_{\ell,m}$.
Again, observe that $\alphamax[f,N_{\ell,m}]$ and $\Deltamax[f,N_{\ell,m}]$ depend only
on the values of $f$ on~$N_{\ell,m}$.

\nicebreak
\begin{thm}\label{Theorem_subsets}
Consider any one of the following three possibilities:
\begin{itemize}\itemsep=0pt
\item[\emph{(i)}]
$\X$ is a uniformly random element of $B_{n,m}$, where $m\leq n/2$. 
\item[\emph{(ii)}]
$\X=(X_1,\ldots,X_\ell)$ is a $N_{\ell,m}$-valued random variable
with the hypergeometric distribution with parameters
$n_1,\ldots, n_\ell\geq 0$
such that $n_1 + \cdots + n_\ell = n\geq 2m$; 
that is,
 \[
    \Prob(\X=(x_1,\ldots,x_\ell)) = 
       \binom{n}{m}^{\!\!-1}\prod_{j=1}^\ell \binom{n_j}{x_j},
             \qquad (x_1,\ldots,x_\ell)\in N_{\ell,m}.
 \]
\item[\emph{(iii)}]
$\X=(X_1,\ldots,X_\ell)$ is a $N_{\ell,m}$-valued random variable
with the multinomial distribution with parameters $p_1,\ldots, p_\ell > 0$
such that $p_1 + \cdots + p_\ell = 1$;  that is,
 \begin{equation}
    \Prob(\X=(x_1,\ldots,x_\ell)) = 
      m! \,\prod_{j=1}^\ell \frac{p_j^{x_j}}{x_j!},
             \qquad (x_1,\ldots,x_\ell)\in N_{\ell,m}.
\label{multinomial}
 \end{equation}
\end{itemize}
With $\varLambda=B_{n,m}$ or $\varLambda=N_{\ell,m}$, and given
a function $f:\varLambda\rightarrow \Complexes$, let
$\alphamax=\alphamax[f,\varLambda]$ and $\Deltamax = \Deltamax[f,\varLambda]$.
Then
 \begin{itemize}\itemsep=0pt
  \item[\emph{(a)}] $\displaystyle
  \E e^{f(\X)} = e^{\E f(\X)}(1+ K(f))$, 
  where $K(f)\in\Complexes$ satisfies
  $\abs{K(f)} \leq e^{\frac18 m\, \alphamax^2 } - 1$.
\item[\emph{(b)}] $\displaystyle 
  \E e^{f(\X)} = e^{\E f(\X) + \tfrac12 \V f(\X) }(1+ L(f)\, e^{\tfrac12 \Var\Im f(\X)})$, where $L(f)\in\Complexes$ satisfies
\[  
   \abs{L(f)} \leq \exp
	\( 
	 \tfrac{1}{2}  m \alphamax^3 + \tfrac{1}{6} m^2 \alphamax^2 \Deltamax 
           + 2 m \alphamax^4 + \tfrac{5}{8}  m^3 \alphamax^2 \Deltamax^2
	 \)-1.
\]
\end{itemize}
\end{thm}

\begin{proof}
First suppose that $\X$ has the distribution described in (i), 
and define $\Tilde{f}:S_n\rightarrow \Complexes$ by
$$
  \Tilde{f}(\omega_1, \ldots, \omega_n) = f(\{\omega_1, \ldots, \omega_m\}), \ \ \ \omega \in S_n.
$$
Let $\Y$ be a uniformly random element of $S_n$. Observe that
$f(\X)$ and $\Tilde{f}(\Y)$ have the same distribution, and hence
\[
 \E e^{f(\X)} = \E e^{\Tilde{f}(\Y)}, \qquad  \E f(\X) = \E \Tilde{f}(\Y),
\]
and similarly for $\V f(\X)$ and $\Var \Im f(\X)$.

Let $\alpha_j= \alpha_j[\Tilde{f},S_n]$ and $\varDelta_{jk} = \varDelta_{jk}[\Tilde{f},S_n]$ denote the
parameters
used in Lemma~\ref{Lemma_Perm},
 defined with respect to the function $\Tilde{f}$ and set $S_n$.
We will apply Theorem~\ref{RanPermThm} to the function $\Tilde{f}$.
Then the bound (a) follows immediately from Theorem~\ref{RanPermThm} (a), since
\[ 
	  \alpha_j \leq \begin{cases} \alphamax, & \text{ for $j=1,\ldots, m$,}\\
                                                           0, & \text{ for $j =m+1,\ldots, n$.}
  \end{cases}
\]
Next, note that $\beta_j=0$ for  $j=m+1,\ldots, n$, and $\varDelta_{jk}=0$ if $k > m$.
We now estimate $\varDelta_{jk}$ when $j < k \leq m$. 

If $b\leq m$ or $a\leq m$ then $D^{(k\, b)}\, D^{(j\, a)} \Tilde{f}=0$, since $\Tilde{f}$
depends only on the set of the first $m$ components of the input permutation.
Next, observe that if $a=b > m$ then
\[ \| D^{(k\, a)}\, D^{(j\, a)} \Tilde{f}\| \leq 2\alphamax,\]
while if $a\neq b$ and $a,b > m$ then
\[ \| D^{(k\, b)}\, D^{(j\, a)} \Tilde{f}\| \leq \Deltamax.\]
Therefore
\[
  \varDelta_{jk} \leq \frac{ 2(n-m)\alphamax + (n-m)(n-m-1)\Deltamax } {(n-j)(n-k)}\, 
              \leq \frac{2}{n-m}\alphamax + \Deltamax.
\]
Hence using Lemma \ref{Lemma_Perm},  it follows that
\[
\beta_j \leq  (m-j) \alphamax\left(\dfrac{2}{n-m}\alphamax + \Deltamax\right)
\]
for $j=1,\ldots, m$. 
Using these bounds and the fact that $2m\leq n$, we find that
\begin{align*}
& \sum_{j=1}^{n-1} \,\(\tfrac{1}{6} \alpha_j^3 + \tfrac{1}{3}\alpha_j \beta_j + \tfrac{5}{8} \alpha_j^4 + \tfrac{5}{8}
\beta_j^2\)\\
&{\quad}\leq \tfrac{1}{6} m\alphamax^3 + \tfrac{1}{3}\, \alphamax^2 \left(\dfrac{2}{n-m}\alphamax + \Deltamax\right)\,
\sum_{j=1}^m (m-j)
  + \tfrac{5}{8} m \alphamax^4 \\
 & \hspace*{3cm} {} + \tfrac{5}{8} \alphamax^2\, 
                       \left(\dfrac{2}{n-m}\alphamax + \Deltamax\right)^2\, \sum_{j=1}^m (m-j)^2\\
&{\quad}\leq \tfrac{1}{2} m\alphamax^3  + \tfrac{1}{6} m^2 \alphamax^2\Deltamax +
+ \tfrac{5}{8} m\alphamax^4  + \tfrac{5}{24} m^3\, \alphamax^2\left(\dfrac{2}{n-m}\alphamax + \Deltamax\right)^2\\&{\quad}\leq 
	 \tfrac{1}{2}  m \alphamax^3 + \tfrac{1}{6} m^2 \alphamax^2 \Deltamax 
           + 2 m \alphamax^4 + \tfrac{5}{8}  m^3 \alphamax^2 \Deltamax^2.
\end{align*}
We used the inequality 
$(a+b)^2 \leq \tfrac{3}{2} a^2 + 3b^2$ in the final line.
Applying Theorem~\ref{RanPermThm}(b) completes the proof for when $\X$ has the distribution described
in (i).

Next, suppose that $\X$ has the hypergeometric distribution described in (ii).  
Take disjoint sets $A_1,\ldots,A_\ell$ with $\card{A_j}=n_j$ for each~$j$.  If we choose a random subset
 $B\subseteq A_1\cup\cdots\cup A_\ell$ with size $m$ then
 $\X=(\card{B\cap A_1},\ldots,\card{B\cap A_\ell})$ has the required distribution.
 Now we can consider $f(\X)$ as a function of $B$ and apply case (i).
 
 Finally, suppose that $\X$ has the multinomial distribution described in (iii). 
Apply case (ii) with $n_j=\lceil p_j t\rceil$ and let $t\to\infty$.
\end{proof}

We remark that by giving tighter bounds on factors of the form $m/(n-m)$ in the above proof,
the constants in the error term $|L|$ for (b) can be improved.
We do not pursue this here.

\nicebreak
\section{Moment calculations}\label{s:moments}

Now we prove a lemma that will be used repeatedly
in the following sections.

\begin{lemma}\label{ranper}
Suppose $u$, $v:\{ 1,2,\ldots, n\} \rightarrow \Reals$.
Define the function $\Psi = \Psi_{u,v}: S_n \to \Reals $ by
  $\Psi(\sigma)=\sum_{j=1}^n u(j) v({\sigma_j)}$ for $\sigma\in S_n$.
  Let $\X=(X_1,\ldots,X_n)$ denote a random permutation uniformly chosen from~$S_n$.
  Define $\bar u=\frac{1}{n} \sum_{j=1}^n u(j)$,
  $\bar v=\frac{1}{n} \sum_{j=1}^n v(j)$. 
Finally, let
\[  \alpha=\(\max_j u(j)-\min_j u(j)\)\,\(\max_j v(j)-\min_j v(j)\).\]
\begin{itemize}\itemsep=0pt
\item[\emph{(i)}]
  Then
\[
    \E \Psi(\X) =n\, \bar u\, \bar v \,\,\, \text{ and } \,\,\,
     \E e^{\Psi(\X)} = e^{\E \Psi(\X) +\frac12\Var \Psi(\X) + L}
\]
   for some $L\in\Reals$ with $\abs L \leq \frac32n\alpha^3+ 11\,n\alpha^4$.
\item[\emph{(ii)}]
Now let $u'$, $v'\in \{ 1,2,\ldots, n\}\rightarrow \Reals $ and
let $\bar{u}'$, $\bar{v}'$ be the average value of $u'$, $v'$, respectively.
Let $\Psi' = \Psi_{u',v'}$.  Then
\[
   \myCov{\Psi(\X),\Psi'(\X)} = \frac{1}{n-1}
      \sum_{j=1}^n (u(j)-\bar u)(u'(j)-\bar u')
       \sum_{k=1}^n (v(k)-\bar v)(v'(k)-\bar v').
\]
In particular,
\[
   \Var \Psi(\X) = \frac{1}{n-1}
      \sum_{j=1}^n (u(j)-\bar u)^2
       \sum_{k=1}^n (v(k)-\bar v)^2.
\]
\item[\emph{(iii)}]
For distinct $j,k\in \{1,\ldots, n\}$, define $E_{jk}:S_n\to\Reals $ by
\[ E_{jk}(\sigma) = (u(j) + v(\sigma_j))(u(k) +  v(\sigma_k)).
\]
Then
\[
 \E E_{jk}(\X) = (u(j) + \bar{v})( u(k) + \bar{v}) - \frac{1}{n(n-1)}\sum_{i=1}^n (v(i) - \bar{v})^2.
\]
\item[\emph{(iv)}]  For $j,k,\ell,m\in \{1,\ldots, n\}$ with $j,k$ distinct and $\ell,m$ distinct,
\[
  \mYCov{E_{jk}(\X), E_{\ell m}(\X)} = \begin{cases} O((\| u\| + \| v\| )^4/n), & 
      \text{ if $\{j,k\}\cap \{\ell,m\} = \emptyset$}; \\
                O( (\| u\| + \|v\|)^4 ), & \text{ otherwise.} \end{cases}
\]
\item[\emph{(v)}] For distinct $j,k\in\{1,\ldots, n\}$,
\begin{align*}
&  \mYCov{E_{jk}(\X),\Psi'(\X)} \\
& = 
\dfrac{1}{n}\,
\Bigl( (u'(j) - \bar{u}')(u(k) + \bar{v}) + (u'(k) - \bar{u}')(u(j) + \bar{v})\Bigr)\,
 \sum_{a=1}^n (v(a)-\bar{v})(v'(a)-\bar{v}')
\\[-1ex]
 & \hspace*{3cm} {}
   + O\left(\frac{(\| u\| + \| v\|)^2 \| u'\|\| v'\|}{n}\right)\\
 &=  O\left((\| u\| + \| v\|)^2 \| u'\|\| v'\|\right).
\end{align*}
\end{itemize}
\end{lemma}
\begin{proof}
We calculate that
\begin{align*}
\E \Psi(\X) &= \sum_{j=1}^n u(j)\, \E v(X_j) = \bar{v}\, \sum_{j=1}^n u(j) = n \bar{u}\bar{v}.
\end{align*}
Next we apply 
 Theorem~\ref{goodthm}(b) and Lemma~\ref{Lemma_Perm} to the Doob martingale for $\Psi:S_n \rightarrow \Reals$, 
as defined in (\ref{Doob_permutations}).  Observe that, for $1\leq j<a\leq n$
we have
\begin{align*}
  D^{(j\, a)} \Psi(\sigma) &= (u(j) - u(a))(v(\sigma_j) - v(\sigma_a))
\end{align*}
Therefore, $\|D^{(j\, a)} \Psi \|\leq \alpha$ and $\alpha_j[\Psi,S_n] \leq \alpha$.
When $1\leq j,k,a,b\leq n$ are distinct, observe that $D^{(k\, b)}\, D^{(j\, a)} \Psi(\sigma) = 0$. 
Otherwise, we can bound
\[
 \|D^{(k\, b)}\, D^{(j\, a)} \Psi\| \leq 2 \|D^{(j\, a)} \Psi\| \leq 2\alpha,
\]
which leads to the estimate $\varDelta_{jk}[\Psi,S_n] \leq 4 \alpha/(n-j)$.
Applying Theorem \ref{RanPermThm} and observing 
that  $\beta_j \leq 4\alpha^2$ gives the stated bound on $L$. This completes the proof of (i).

For (ii), we may assume without loss of generality that $\bar{u}$, $\bar{v}$, $\bar{u}'$,
$\bar{v}'$ all equal, by shifting $u$, $v$, $u'$, $v'$
if necessary.  This shifts the distributions of~$\Psi$ and~$\Psi'$ but has no effect
on their covariance.

Next observe that for $j,k = 1,\ldots, n$,
\[ \mYCov{u(j)v(X_j), u'(k) v'(X_k)} =
   \biggl(\sum_{i=1}^n v(i) v'(i) \biggr) \frac{u(j) u'(k)}{n}
   \biggl( 1  - \frac{(n+1)\, \mathbf{1}_{j\neq k}}{n-1}\biggr)
\]
where $\mathbf{1}_{j\neq k}$ is the indicator variable
 which equals 1 when $j\neq k$ and 0 otherwise.
Summing this expression over all pairs $(j,k)$ proves the first statement of (ii),
and replacing $\Psi'$ by $\Psi$ completes the proof of (ii).

For part (iii), we calculate that 
\[ 
  \myE{ v(X_1) v(X_2)} = \bar{v}^2 - \frac{1}{n(n-1)} \sum_{i=1}^n (v(i) - \bar{v})^2,
\]
from which (iii) follows. 

For (iv), it is not difficult to prove by induction on $k$ that 
\begin{equation}
\label{induction}
 \myE{ v(X_1) v(X_2)\cdots v(X_k)} = \bar{v}^k + O(n^{-1})\, \|v\|^k,
 \quad\text{for $k=O(1)$}.
\end{equation}
This follows using the fact that, for $k\geq 1$,
\[ 
  \mYE{ v(X_1) v(X_2)\cdots v(X_k) v(X_{k+1})} = 
    \dfrac{1}{n}\, \sum_{i=1}^n \mYE{v(X_1) v(X_2)\cdots v(X_k) \St X_{k+1} = i}\, v(i),
\]
after observing that the average of $\{ v(j) \mid j\neq i\}$ equals $\bar{v} + O(\|v \|/n)$.
It follows that
\begin{equation}
\label{EEjk-bound} \E E_{jk}(\X) = (u(j) + \bar{v})(u(k) + \bar{v}) + O\left(\frac{ \|v\|^2}{n}\right).
\end{equation}
Therefore $E_{jk}(\X)$, $E_{\ell m}(\X)$, $\E E_{jk}(\X)$ and $\E E_{\ell m}(\X)$ are all
$O((\|u \| + \| v\|)^2)$, from which we conclude that $\myCov{E_{jk},E_{\ell m}} = O((\| u\| + \|v\|)^4)$,
for any distinct $j,k$ and distinct $\ell, m$.
In the case that $\{ j,k\}\cap \{\ell,m\}=\emptyset$,  it follows from (\ref{induction})
that
\[ \myE{ E_{jk}(\X) E_{\ell m}(\X)} = 
  (u(j) + \bar{v})(u(k) + \bar{v})(u(\ell) + \bar{v})(u(m) + \bar{v}) 
    + O\left(\frac{(\| u\| + \|v\|)^4}{n}\right).
\]  
The improved bound on $\myCov{E_{jk}(\X),E_{\ell m}(\X)}$ follows directly.

Finally, we prove part (v).  First we calculate
$\myCov{E_{jk}(\X),v'(X_i)}$ under the assumption that $i\not\in \{j,k\}$.
In this case, 
\begin{align*}
& \myE{E_{jk}(\X)\,v'(X_i)}\\
 &= \dfrac{1}{n} \sum_{a=1}^n \myE{(u(j) + v(X_j))(u(k) + v(X_k))\mid X_i = a}\, v'(a)\\
&= \dfrac{1}{n}\sum_{a=1}^n v'(a) \biggl( \Bigl(u(j) + \frac{n\bar{v} - v(a)}{n-1}\Bigr)\Bigl(u(k) + \frac{n\bar{v} -
v(a)}{n-1}\Bigr) \\
         &{\kern16em} - \frac{1}{(n-1)(n-2)}\sum_{b:b\neq a} (v(b) - \bar{v})^2 \biggr)\\
     &= \bar{v}' (u(j) + \bar{v})(u(k) + \bar{v}) -\frac{\bar{v}'}{n(n-1)}\sum_{b=1}^n (v(b) - \bar{v})^2\\
    & \hspace*{4cm} {} - \frac{u(j) + u(k) + 2\bar{v}}{n(n-1)}\, 
     \sum_{a=1}^n v'(a) (v(a) - \bar{v})  + O\left(\frac{\| v\|^2 \| v'\|}{n^2}\right).
\end{align*}
The second line follows from applying (iii) to the restriction of $u$, $v$
to $\{ 1,\ldots, n\}\setminus \{ a\}$.
Subtracting $\E E_{jk}\, \E v'(X_i)$ using (iii), we find that
\[ \mYCov{E_{jk}(\X), v'(X_i)} = -\frac{u(j)+u(k) + 2\bar{v}}{n(n-1)}\, 
               \sum_{a=1}^n (v(a)-\bar{v})(v'(a) - \bar{v}')
 + O\left(\frac{\|v\|^2 \|v'\|}{n^2}\right).
\]

Now suppose that $i\in\{j,k\}$.  When $i=j$, using similar calculations as above, we obtain
\[
\myE{E_{jk}(\X)v'(X_j)}= \frac{1}{n} \sum_{a=1}^n (u(j) + v(a))(u(k) + \bar{v})\, v'(a)
        + O\left(\frac{(\| u\| + \|v\|)^2 \|v'\|}{n}\right)\]
and hence
\[ \mYCov{E_{jk}(\X), v'(X_j)} = (u(k)+\bar{v}) 
     \frac{1}{n} \sum_{a=1}^n (v(a) - \bar{v})( v'(a) - \bar{v}')  
   + O\left( \frac{(\|u\| + \|v\|)^2 \| v'\|}{n}\right).
\]
A similar formula holds when $i=k$.  Now summing over all $i$, we obtain the stated
formula for $\myCov{E_{jk}(\X),\Psi'(\X)}$, completing the proof.
\end{proof}


\nicebreak
\section{Subgraphs isomorphic to a given graph}\label{s:subgraph}

If $H$ is a graph on the vertex set $\{1,\ldots, n\}$, let $\Prob(\dvec,H)$ be the probability that
$G\sim\Gd$ contains $H$ as a subgraph.
The starting point for our arguments is the following result adapted from
 McKay~\cite[Theorem~2.1]{ranx}.  
We state it using the parameters defined in (\ref{R-lambda-delta}) and (\ref{Mt}).

\begin{thm}\label{treeprob}
For any constant $\eta\in(0,\frac12)$  there is some
$\eps_5(\eta) > 0$ such that the following holds
 for every fixed $\eps \in (0,\eps_5(\eta)]$.
Let $\dvec$ be a degree sequence which satisfies~\eqref{curlyA}.
Suppose that $H$ is a graph on the vertex set $\{ 1,\ldots, n\}$ with
$m\leq n^{1+2\eps}$ edges and degree
sequence $\hvec = (h_1,\ldots, h_n)$ 
such that $\norm{\hvec} \leq n^{1/2 + \eps}$.
Then
\begin{equation}
\label{p-expression}
   \Prob(\dvec,H) = \lambda^m \, \exp\( f(\dvec,\hvec) + g(\dvec,H) + O(n^{-1/2+\eta})\),
\end{equation}
 where
\begin{align*}
   f(\dvec,\hvec) &= \frac{(1-\lambda)}{4\lambda}(\mu_1^2 + 2 \mu_1 - 2\mu_2)                 
                   - \frac{(1-\lambda^2)}{6\lambda^2 n} \, \mu_3 
+ \frac{1}{\lambda n}
                     \sum_{j=1}^n (d_j-d)h_j
\\ &{\kern6em}   
                     + \frac{1}{2\lambda^2 n^2} \sum_{j=1}^n (d_j - d) h_j^2
                     - \frac{1}{2\lambda^2 n^2} \sum_{j=1}^n (d_j-d)^2 h_j, \\
   g(\dvec,H) &= - \frac{1}{\lambda(1-\lambda)n^2}
                             \sum_{jk\in E(H)} (d_j - d-h_j+\lambda h_j)
                                                    ( d_k - d-h_k+\lambda h_k).
\end{align*}
\end{thm}
\begin{proof} Theorem~2.1 in~\cite{ranx} is stated slightly differently.
It supposes two constants $a,b>0$ with $a+b<\frac12$ and the second
part of~\eqref{curlyA} reads
\[   \min\{d,n-d-1\} \geq \frac{n}{3a \log n}. \]
If $\eps=\eps(a,b)$ defined in \cite[Theorem~2.1]{ranx} 
 then \eqref{p-expression} holds with the error term $O(n^{-b})$.
To obtain our formulation, take $a=2\eta^2$, $b=\frac12-\eta$ 
and $\eps_5(\eta) = \eps(a,b)$. 
Clearly, for any $\eps<\eps_5(\eta)$ formula   \eqref{p-expression} also holds with the same error term, since all assumptions depending on $\eps$ become more strict.
\end{proof}

\begin{lemma}\label{uniformity}
Let $A=\bigcup_{k=1}^\infty A_k\subseteq\Complexes$ be such that,
for each $k$,
$A_k\ne\emptyset$  and $\sup_{a\in A_k} \abs{a} <\infty$.
Suppose that $a_k=O(1)$ as $k\to\infty$ for every sequence
$a_1,a_2,\ldots$ with $a_k\in A_k$ for each~$k$.
Then $\sup_{a\in A} \,\abs{a} < \infty$.
\end{lemma}
\begin{proof}
For each $k$ there is some
$a'_k\in A_k$ such that $\abs{a'_k} \geq \sup_{a\in A_k} \abs{a}-1$.
Now we have
$\sup_{a\in A} \abs{a} \leq \sup_{k\geq 1} \abs{a'_k}+1<\infty$. 
\end{proof}

\begin{remark}\label{uniformremark}
Lemma~\ref{uniformity} will be useful whenever we need to sum
Theorem~\ref{treeprob}, or similar theorems, over many values of
the parameters.  For fixed $\eta$, there are only finitely many
degree sequences $\dvec$ and graphs $H$ for each~$n$ that
satisfy the conditions.  Therefore, applying Lemma~\ref{uniformity}
to the sets consisting of the error terms in Theorem~\ref{treeprob}   
for these $\dvec$ and $H$, scaled by a factor $n^{1/2-\eta}$,
shows that
the error term is uniform over $\dvec$ and~$H$.  That is, there is
a function $C(\eta)$,
not depending on any other parameters, such that the absolute
value of the error term is bounded above by $C(\eta)n^{-1/2+\eta}$.
\end{remark}

In order to compute the expected number of subgraphs isomorphic to $H$,
we must sum $\Prob(\dvec,H)$ over all possible locations of $H$.
We will find it convenient to average over permutations of
the degree sequence $\dvec$ rather than over labellings of the
subgraph $H$; by symmetry, this is equivalent.

For a permutation $\sigma\in S_n$, let
$\dvec^{\,\sigma}=(d_{\sigma_1},\ldots, d_{\sigma_n})$
denote the permuted degree sequence, and let
\begin{align*}
   f_{\hvec}(\sigma) &= \frac{(1-\lambda)}{4\lambda}(\mu_1^2 + 2 \mu_1 - 2\mu_2)                 
- \frac{(1-\lambda^2)}{6\lambda^2 n} \, \mu_3 
 +  \frac{1}{\lambda n} \sum_{j=1}^n\, (d_{\sigma_j}-d)h_j \\ 
             &\qquad {} 
                                          + \frac{1}{2\lambda^2 n^2} \sum_{j=1}^n\, (d_{\sigma_j} - d) h_j^2
                     - \frac{1}{2\lambda^2 n^2} \sum_{j=1}^n\, (d_{\sigma_j}-d)^2 h_j, \\
   g_{H}(\sigma) &=  - \frac{1}{\lambda(1-\lambda)n^2}
                             \sum_{jk\in E(H)} \! (d_{\sigma_j} - d-h_j+\lambda h_j)
                                                    ( d_{\sigma_k} - d-h_k+\lambda h_k).
\end{align*}
Since $f_{\hvec}(\sigma) = f(\dvec^\sigma,\hvec)$ and 
$g_{H}(\sigma) = g(\dvec^\sigma,H)$,
Theorem~\ref{treeprob} implies that
the expected number of subgraphs isomorphic to $H$ in a uniformly random graph
with degree sequence $\dvec$ is
\begin{equation}
    (1 + O(n^{-1/2+\eta})) \,\frac{n!}{|\Aut(H)|}\, \lambda^m\, \mYE{
               \exp( f_{\hvec}(\X) + g_H(\X)) },
\label{start-point}
\end{equation}
where the expectation is taken with respect to a uniformly random element $\X$ of $S_n$.
Here we have used the uniformity of the error term $O(n^{-1/2+\eta})$ in
Theorem~\ref{treeprob}, as explained in Remark~\ref{uniformremark}.

Define 
\begin{equation}\label{def_eps1}
	\eps_1(\eta) = \min\left\{\eps_5(\eta),\tfrac1{12}\eta\right\},
\end{equation}
where $\eps_5(\eta)$ is provided by Theorem~\ref{treeprob}.
Before proving Theorem~\ref{subcount}, we apply the results of Section~\ref{s:moments}
to obtain the following expressions.

\begin{lemma}\label{allvar}
Let $\eta\in(0,\frac12)$ be constant. 
If assumptions \eqref{curlyA} and \eqref{assumption}
hold with  $\eps \in (0,\eps_1(\eta)]$, then
\begin{align*}
     \E f_{\hvec}(\X) 
    &= \frac{(1-\lambda)}{4\lambda}(\mu_1^2 + 2 \mu_1 - 2\mu_2)                 
                   - \frac{(1-\lambda^2)}{6\lambda^2 n} \, \mu_3 
  -\frac{ R }{2\lambda^2 n}\mu_1,\\
   \E g_{H}(\X) 
     &= - \frac{1-\lambda}{\lambda n^2} \sum_{jk\in E(H)} h_j h_k + O(n^{-1/2+\eta}), \\
\mYVar{f_{\hvec}(\X)& {} +g_{H}(\X)} =
   \frac{R }{\lambda^2 n}(\mu_2  - \mu_1^2)
   + O(n^{-1/2+\eta}).
\end{align*}
\end{lemma}

\begin{proof}
We repeatedly use the following bounds in our estimates: 
\[ 
  n \mu_t \leq 2 \, \|h\|^{t-1}\, m  \leq 2n^{(t+1)(1/2 + \eps)} \quad \text{ and }  
    \quad R \leq \delta^2 \leq n^{1+2\eps}.
\]

The expression for $\E f_{\hvec}(\X)$ follows directly from applying
Lemma~\ref{ranper}(i) to the terms of $f_{\hvec}(\X)$. 
(The first two terms are constants, the third and fourth term both give zero
since the average of $d_{\sigma_j}-d$ is zero, and the fifth term of $f_{\hvec}$ provides
the final term in the expected value.)
Similarly, using Lemma~\ref{ranper}(iii) and \eqref{curlyA} we have
\[
     \E g_{H}(\X) 
     = - \frac{1-\lambda}{\lambda n^2} \sum_{jk\in E(H)} h_j h_k
      + \frac{R}{\lambda(1-\lambda)n^2(n-1)},
\]
which matches the given expression after applying the assumptions.

Recall that for real random variables $X_1,\ldots,X_t$
we have
\[
    \MYVar{\,\sum_{j=1}^t X_j} = \sum_{j,k=1}^t\myCov{X_j,X_k}.
\]

For all positive integers $i,\ell$, and for all $jk\in E(H)$,
define the functions $\Psi^{(i,\ell)}, E_{jk}:S_n\rightarrow \Reals $ by
\begin{align*} \Psi^{(i,\ell)}(\sigma) &= \sum_{j=1}^n \,(d_{\sigma_j} - d)^i\, h_j^\ell,
 \\
   E_{jk}(\sigma) &= \((\lambda - 1)h_j + d_{\sigma_j} - d\)
      \((\lambda - 1)h_k + d_{\sigma_k} - d\)
\end{align*}
for all $\sigma\in S_n$.
Using Lemma~\ref{ranper}(ii) with $u(j) = h_j$ and $v({\sigma_j}) = d_{\sigma_j}-d$, we find that
\[ \MyVar{\dfrac{1}{\lambda n}\Psi^{(1,1)}(\X)} 
= \frac{R(\mu_2 - \mu_1^2)}{\lambda^2 (n-1)}
= \frac{R(\mu_2 - \mu_1^2)}{\lambda^2 n}  + O(n^{-1/2+\eta}).
\]

Applying \eqref{assumption} and Lemma~\ref{ranper}(v) with
$u(j) = (\lambda-1)h_j$, $u'(j) = h_j$, and $v(j) = v'(j) = d_j - d$, we obtain
\begin{align*}
  \sum_{jk\in E(H)}&\mYCov{ \Psi^{(1,1)}(\X),E_{jk}(\X)} \\
    &= O\( \mu_1 (\norm u+\norm v)^2 \norm{u'}\norm{v'}\)
          + R (\lambda-1) \sum_{jk\in E(H)} \( (h_j-\mu_1)h_k + (h_k-\mu_1)h_j \) \\
    &= O(n^{2+6\eps}).
\end{align*}
Consequently,
\[
    \MyCov{\dfrac{1}{\lambda n}\Psi^{(1,1)}(\X),
          -\dfrac{1}{\lambda(1-\lambda)n^2} E_{jk}(\X)}
  = O(n^{-1/2+\eta}).
\]

Observe also that
\begin{align}
& \MYVar{-\dfrac{1}{\lambda(1-\lambda)n^2}\sum_{jk\in E(H)} E_{jk}(\X)} \notag\\
  &= \frac{1}{\lambda^2 (1-\lambda)^2 n^4}\, \sum_{jk\in E(H)}\, \sum_{i\ell\in E(H)} 
    \myCov{E_{jk}(\X), E_{i\ell}(\X)} \notag\\
  &= \frac{1}{\lambda^2(1-\lambda)^2 n^4}\, \left( m\norm{\hvec} O\(( (1-\lambda)\norm{\hvec} + \delta)^4\)
                         +  m^2\, O\left(\frac{( (1-\lambda)\norm{\hvec} + \delta)^4}{n}\right)\right) \notag\\
  &= O(n^{-1/2+\eta}). \label{Evar}
\end{align}
This follows from Lemma~\ref{ranper}(iv), using the fact that there are at most $m\norm{\hvec}$ pairs of
adjacent edges of $H$ and at most $m^2$ pairs of non-adjacent edges of $H$.

We can now verify that all the remaining contributions to
$\myVar{f_{\hvec}(\X) +g_{H}(\X)}$
are $O(n^{-1/2+\eta})$.
Using Lemma~\ref{ranper}(ii), we have
\begin{align}
 \MyVar{\dfrac{1}{2\lambda^2 n^2}\Psi^{(1,2)}(\X)}
 &=  \frac{(\mu_4-\mu_2^2)R}{4\lambda^4n^2(n-1)} = O(n^{-1/2+\eta}) \label{Bvar}\\
  \MyVar{ -\dfrac{1}{2\lambda^2n^2}\Psi^{(2,1)}(\X)}
 &= \frac{(\mu_2-\mu_1^2)(R_4-R^2)}{4\lambda^4n^2(n-1)} = O(n^{-1/2+\eta}), \label{Cvar}
\end{align}
where $R_4=\sum_{j=1}^n (d_j-d)^4$.
For any two real random variables $Z,Z'$ we have
$\Abs{\myCov{Z,Z'}} \leq \max\{\Var Z, \Var Z'\}$, so the three
covariances involving two of the quantities in~(\ref{Evar}--\ref{Cvar}) are also
$O(n^{-1/2+\eta})$.
The only remaining covariances are
\begin{align*}
   \MyCov{\dfrac{1}{\lambda n}\Psi^{(1,1)}(\X), \dfrac{1}{2\lambda^2 n^2}\Psi^{(1,2)}(\X)}
   &= \frac{R(\mu_3 - \mu_1\mu_2)}{2\lambda^3 n(n-1)} 
      = O(n^{-1/2+\eta}) \text{~~and}\\
   \MyCov{\dfrac{1}{\lambda n} \Psi^{(1,1)}(\X), -\dfrac{1}{2\lambda^2n^2}\Psi^{(2,1)}(\X)}
   &= -\frac{(\mu_2-\mu_1^2)\sum_{j=1}^n (d_j-d)^3}{2\lambda^3n^2(n-1)}
   = O(n^{-1/2+\eta}),
\end{align*}
using~\eqref{assumption}, 
since $R\leq \delta^2$ and $\mu_1^2\leq\mu_2\leq\mu_3$. 
This completes the proof.
\end{proof}

\begin{proof}[Proof of Theorem~\ref{subcount}]

Define $\eps_1(\eta)$ as in \eqref{def_eps1}.
We will apply Theorem~\ref{RanPermThm} to estimate (\ref{start-point}). 
 The expected value and variance of $f_{\hvec} + g_H$ are given in
Lemma~\ref{allvar}.  It remains to prove that 
\[ \sum_{j=1}^{n-1} \,(\nfrac{1}{6} \alpha_j^3 + \nfrac{1}{3} \alpha_j \beta_j + \nfrac{5}{8} \alpha_j^4 + \nfrac{5}{8} \beta_j^2
   \)
  = O(n^{-1/2+\eta}),
\]
where $\alpha_{j} = \alpha_j[f_{\hvec} + g_H,S_n]$,  $\beta_j = \sum_{k=j+1}^{n-1} \alpha_k\varDelta_{jk}$ and $\varDelta_{jk} = \varDelta_{jk}[f_{\hvec} + g_H,S_n]$.

Without loss of generality, we can assume that
\[ 
h_1 \geq h_2 \geq \cdots \geq h_n.
\]
We calculate that, for $1\leq j< a\leq n$,
\[ \| D^{(j\, a)} f_{\hvec}\| = O\left(\frac{\delta h_j}{\lambda n} \right).\]
Therefore, $\alpha_j[ f_{\hvec},S_n] =  O\left(\dfrac{\delta h_j}{\lambda n} \right)$.
Observe also that 
$ \| D^{(k\, b)}\, D^{(j\, a)} f_{\hvec}\| = 0$
whenever $j,k,a,b$  are distinct. Otherwise, for $1\leq j< a\leq n$ and
$1\leq k < b\leq n$ with $j,k,a,b$ not all distinct, we use the bound 
\[
\| D^{(k\, b)}\, D^{(j\, a)} f_{\hvec}\| \leq 2 \| D^{(j\, a)} f_{\hvec}\| = O\left(\frac{\delta h_j}{\lambda n}
\right).
\]
Thus, $\varDelta_{jk}[ f_{\hvec},S_n] =  O\left(\frac{\delta h_j}{\lambda n(n-j)} \right)$.

Next we need to consider $g_{H}$, and calculate that
\[ \| D^{(j\, a)} g_{H}\| = O\left(\frac{(\delta+h_j)^2 h_j}{\lambda(1-\lambda) n^2}\right), \]
since only the terms of $g_H$ corresponding to edges incident with $j$ or $a$ can contribute.
This  gives us 
$ 
\alpha_j[g_H,S_n] = O\left(\dfrac{(\delta+h_j)^2 h_j}{\lambda(1-\lambda) n^2}\right) = O\left(n^{-1/2 + 4\eps}\right).
$

Suppose that $j,k,a,b$ are all distinct, with $j<a$ and $j<k<b$.
If $\{ jk,\, jb,\, ab,\, kb\}\cap E(H) = \emptyset$ then
\[ \| D^{(k\, b)}\, D^{(j\, a)} g_{H}\| = 0,\]
and otherwise
\[ \| D^{(k\, b)}\, D^{(j\, a)} g_{H}\| = O\left(\frac{(\delta  + h_j)^2}{\lambda(1-\lambda) n^2}\right).
\]
Therefore,
\[ \varDelta_{jk}[g_H,S_n)] = O\left(\frac{(\delta  + h_{j})^2}{\lambda(1-\lambda) n^2}\right)\,
         \left(\mathbf{1}_{jk\in E(H)} + \frac{h_j}{n-k}\right).\]
To see this, we recall that 
\[
\varDelta_{jk}[g_H,S_n] = \frac{1}{(n-j)(n-k)}
                             \sum_{a=j+1}^n\,\sum_{b=k+1}^n 
  \| D^{(k\, b)} D^{(j\, a)} g_H\| 
\]
 and observe that there are at most $h_j$ choices for $b>k$ such that $jb\in E(H)$. Also, 
 there are at most $n-j$ choices for $a$: dividing the product of these by $(n-k)(n-j)$
leads to the term $h_j/(n-k)$.  Similarly, there are at most $h_k$
choices for $a>j$ such that $ka\in E(H)$, and then at most $n-k$ choices for $b>k$.
If $jk\in E(H)$ then there are at most $(n-k)(n-j)$ choices for $a, b$,
and there are at most $(n-k) h_k$ edges with at least one end-vertex in $\{ k+1,\ldots, n\}$
(this counts the choices for $ab\in E(H)$).
The ``diagonal'' terms (where $a=b$ or $a=k$) satisfy 
\[ 
\| D^{(k\, b)}\, D^{(j\, a)} g_{H}\| \leq 2 \|D^{(j\, a)} g_{H}\| = O\left(\frac{(\delta+h_j)^2h_j}{\lambda(1-\lambda)
n^2}\right).
\]
So their contribution to $\varDelta_{jk}(g_H,S_n)$ is bounded by
$O\left(\dfrac{(\delta + h_{j})^2h_j}{\lambda(1-\lambda) n^2 (n-j)}\right)$.
Combining estimates above and recalling \eqref{triangle_aD}, we  conclude
\begin{align*}
\alpha_{j}  &= O\biggl(\frac{\delta h_j}{\lambda n} + n^{-1/2 + 4\eps}\biggr), \hspace{60mm}  1 \leq j\leq n, \\
\varDelta_{jk} &=  O\biggl( \frac{\delta h_j}{\lambda n(n-j)} + \frac{(\delta  + h_{j})^2}{\lambda(1-\lambda) n^2}\,
         \biggl(\mathbf{1}_{jk\in E(H)} + \frac{h_j}{n-k}\biggr)\biggr),  \ \ \ \ \ 1 \leq j<k\leq n.
\end{align*}

Using the inequality  
$(|x|+|y|)^3 \leq 4(|x|^3 + |y|^3)$ for each term of the sum, we get
\begin{equation}
 \label{Rk3} \sum_{j=1}^{n-1} \alpha_j^3 =
  \sum_{j=1}^{n-1} O\biggl(\biggl(\frac{\delta h_j}{\lambda n} + n^{-1/2 +4 \eps}\biggr)^{\!\!3\,}\,\biggr) =
   O\biggl(\frac{\delta^3 \mu_3}{\lambda^3 n^2} + n^{-1/2+12\eps}\biggr).
\end{equation}
This is $O(n^{-1/2+\eta})$ by (\ref{assumption}) 
and our assumption $\eps\leq \frac1{12}\eta$. We now want to show that the other error terms
from Theorem~\ref{RanPermThm}
all fit inside this bound too.

Now
\[
  \sum_{k=j+1}^{n-1} \mathbf{1}_{jk\in E(H)} \leq h_j,\qquad \qquad \sum_{k=j+1}^{n-1} \frac{1}{n-k} \leq 1+\log n,
\]
so Lemma~\ref{Lemma_Perm} gives
\begin{align*} 
  \beta_j &= 2\sum_{k=j+1}^{n-1} \alpha_k \varDelta_{jk}
= O\left( \frac{\delta^2 h_j^2}{\lambda^2 n^2} + \frac{(\delta+ h_{j})^2 \, \delta\, h_j^2\, \log
n}{\lambda^2(1-\lambda) n^3}\right)\\
   &= O\biggl( \frac{\delta^2 h_j^2}{\lambda^2 n^2} + \frac{\delta h_j^4 \log n}{\lambda^2  (1-\lambda) n^3}  \biggr)
= O\biggl( \biggl(\frac{\delta h_j}{\lambda n} + \frac{h_j^3 \log n}{\lambda (1-\lambda) n^2}  \biggr)^{\!\!2\,\,}\biggr)   \\ &= O\biggl( \biggl( \frac{\delta h_j}{\lambda n} + n^{-1/2+4\eps}\biggr)^{\!\!2\,\,} \biggr).
\end{align*}
From this we find that
\[ \sum_{j=1}^{n-1} \alpha_j \beta_j = \sum_{j=1}^{n-1} O\biggl(\biggl(\frac{\delta h_j}{\lambda n} + n^{-1/2 +4
\eps}\biggr)^{\!\!3\,\,}\biggr)
\]
which is $O(n^{-1/2+\eta})$, similarly to \eqref{Rk3}.

For the two remaining terms, note that
\[
\sum_{j=1}^{n-1} O\biggl(\biggl(\frac{\delta h_j}{\lambda n} + n^{-1/2 +4 \eps}\biggr)^{\!\!4\,\,}\biggr)  = 
\sum_{j=1}^{n-1} O\biggl(\biggl(\frac{\delta h_j}{\lambda n} + n^{-1/2 +4 \eps}\biggr)^{\!\!3\,\,}\biggr)
\]
whenever the RHS is $O(1)$, and furthermore both $\sum_{j=1}^{n-1} \alpha_j^4$ and $\sum_{j=1}^{n-1} \beta_j^2$ are covered by
$O(n^{-1/2+\eta})$.
Applying  Theorem~\ref{RanPermThm}  and using Lemma \ref{allvar} completes the proof.
\end{proof}

\begin{proof}[Proof of Corollary \ref{Cor:subcount}.]
To show \eqref{sub-1}, we recall that, by assumption, 
$\mu_3 \leq \lambda^2 n^{1/2+\eta}$.
Hence
\[
 \frac{1-\lambda}{\lambda n^2} \!\sum_{jk \in E(H)} h_j h_k  \leq 
   \frac{1-\lambda}{ 2\lambda n^2}\!
    \sum_{jk \in E(H)} (h_j^2 + h_k^2)  = 
    \frac{1-\lambda}{2\lambda n} \mu_3 = O(n^{-1/2+\eta}).
\]
To prove the second part of the corollary, we can use the bound
$R \leq \delta^2 \leq n^{1+2\eps}$ and observe that $\mu_1\leq \mu_2 $
since each $h_j$ is a natural number,  and $\mu_1^2 \leq \mu_2$ by the power mean inequality.
\end{proof}

\begin{proof}[Proof of Corollary \ref{Cor:regularsubgraphs}.]
 Part (a) is a simple application of Corollary~\ref{Cor:subcount}.
 
 To prove part (b), note that the argument of the exponential in
 part (a) depends only on the degree~$h$, and not on the 
  finer structure of~$H$, within the error term. 
 Therefore, using the logic in Remark~\ref{uniformremark},
 the expected number of spanning $h$-regular subgraphs is
 \[
    \operatorname{RG}(n,h) \, \lambda^m
    \exp\biggl( -\frac{1-\lambda}{4\lambda} h(h-2) - \frac{Rh}{2\lambda^2n}
               + O(n^{-1/2+\eta}) \biggr),
 \]
 where $\operatorname{RG}(n,h)$ is the number of labelled regular graphs of order~$n$
 and degree $h$.
 From~\cite{bdmreg} we know that
 \[
     \operatorname{RG}(n,h) = \frac{(nh)!}{(nh/2)!\, 2^{nh/2} (h!)^n}
        \exp\biggl( -\frac{h^2-1}{4} + O(h^3/n) \biggr).
 \]
 Now apply Stirling's formula and observe that
$h^3/n\leq n^{-1+6\eps}\leq n^{-1/2+\eta}$, by (\ref{def_eps1}).
\end{proof}


\nicebreak
\section{Spanning trees }\label{s:dense}

As another application of our results, we calculate the expected
number of spanning trees of $G\sim\Gd$  where the degree sequence $\dvec=(d_1,\ldots,d_n)$ 
satisfies~\eqref{curlyA}. Recall the parameters defined in (\ref{R-lambda-delta}).

\subsection{Plan of attack}

For some $\eps>0$, to be chosen later, define
\begin{align*}
   D &= \bigl\{ (h_1,\ldots,h_n)\in \{1,2,\ldots, n-1\}^n 
                 : h_1+\cdots+h_n=2n-2 \bigr\}, \\
   D_\good &= \bigl\{ (h_1,\ldots,h_n)\in D 
                  : h_1,\ldots,h_n\leq n^{3\eps} \bigr\}, \\
   D_\bad &= D \setminus D_\good.
\end{align*}
Let $\calT$ be the set of all labelled trees with $n$ vertices.
For $\hvec\in D$, let $\calT_{\hvec}$ be the set of all $T\in\calT$
with degree sequence~$\hvec$,
and note that every $T\in\calT$ belongs to $\calT_{\hvec}$ for
some $\hvec\in D$.
It is well known that the number of trees with degree sequence $\hvec$ is
\begin{equation}
\label{nr-x}
|\calT_{\hvec}|=  \binom{n-2}{h_1{-}1,\ldots,h_n{-}1}
\end{equation}
(see~\cite[Theorem 3.1]{moon}).
Also define $\calT_\good = \bigcup_{\hvec\in D_\good} \calT_{\hvec}$
and $\calT_\bad = \bigcup_{\hvec\in D_\bad} \calT_{\hvec}$.
A tree is called \emph{good} if it belongs to $\calT_\good$, and otherwise
it is bad.

Our approach will be to write the expected number of spanning trees in
a uniformly random graph with degree sequence $\dvec$ as
\begin{equation} 
  \sum_{T\in \calT} \Prob(\dvec,T) = \sum_{T\in \calT_{\good}} \Prob(\dvec,T) +
                      \sum_{T\in \calT_{\bad}} \Prob(\dvec,T).
\label{approach}
\end{equation}
	Theorem~\ref{treecount}  follows immediately from Lemma \ref{l:good_trees}, which 
	counts good spanning trees, and Lemma \ref{l:bad_trees}, which counts bad spanning trees.
In fact, bad spanning trees will turn out to be rare so the second sum will
contribute a negligible amount relative to the first sum.

\subsection{The expected number of good spanning  trees}\label{s:good-trees}

It will be useful to define a random variable related to the
degree sequence of a tree uniformly chosen from $\calT$, or from $\calT_{\hvec}$.
Since we are only interested in good trees, we will also consider the truncation
of these random vectors, where any entry
larger than $\lfloor n^{3\eps} \rfloor$ is replaced by~$\lfloor n^{3\eps}\rfloor$.

\nicebreak
\begin{lemma}\label{distributions}
Let $\X=(X_1,\ldots,X_n)$ be the degree sequence of
a random tree uniformly chosen from~$\calT$.
 Then, for any fixed $\eps>0$, the following hold.
\begin{itemize}\itemsep=0pt
\item[\emph{(i)}]
 The random vector $(X_1-1,\ldots, X_n-1)$
has a multinomial distribution with parameters $m=n-2$,
$k=n$, $\lambda_1 = \cdots = \lambda_k = 1$, in the notation of \emph{(\ref{multinomial})}.
\item[\emph{(ii)}]
Next, consider a random variable $\Y\in\{0,1,2,\ldots\,\}^n$,
whose components are
i.i.d.\ Poisson variables with mean~1. For each $\yvec$, we have that 
\[ \Prob(X_1=y_1{+}1,\ldots,X_n=y_n{+}1) 
= \Prob(Y_1=y_1,\allowbreak\ldots,Y_n=y_n\st Y_1+\cdots+Y_n=n-2).\]
\item[\emph{(iii)}]
Define the random variable $\Z=(Z_1,\ldots,Z_n)$, where
$Z_j=\min\{X_j, \lfloor n^{3\eps} \rfloor \}$ for each~$j$.
Then uniformly over $j\ne k$,
\begin{gather*}
   \E Z_j   = 2+O(n^{-1}),~~
   \E Z_j^2 = 5+O(n^{-1}), \\
   \Var Z_j = 1+O(n^{-1}),~~
   \Var Z_j^2 = 27 + O(n^{-1}),~~
   \myCov{Z_j,Z_k} = -n^{-1} + O(n^{-2}).
\end{gather*}
\end{itemize}
\end{lemma}

\begin{proof}
Statements (i) and (ii) are well-known and follow easily from~\eqref{nr-x}.

For (iii), note that the probability generating function of $\X$
is 
\[ 
 p(\xvec)=\sum_{\hvec}\,\card{\calT_{\hvec}}\,\xvec^{\hvec}=
  n^{-n+2}\, x_1\cdots x_n(x_1+\cdots+x_n)^{n-2}.
\]
This allows computation of small moments of $\X$, for example
\[ 
   \E X_1 = \frac{\partial}{\partial x_1} p(\xvec)|_{(1,\ldots,1)}
   = 2-\dfrac 2n.
\]
The differences between small moments of $\Z$ and the
corresponding moments of $\X$ are within the given error terms.
To see this, we can set all but one of the arguments of $p(\xvec)$ equal to~1
to find the distribution of the degree of one vertex. Thus we find that
$\Prob(X_1=t)\leq 1/(t-1)!$ for $t\geq 1$ and so
\[
\Prob(\Z\ne\X) \leq n\,\Prob(X_1>\lfloor n^{3\eps}\rfloor) = o(e^{-n^{3\eps}}).
\]
Statement (iii) follows.
\end{proof}

The following result from~\cite[Section 3]{GIKM} will be useful.

\begin{lemma}\label{treeedgeprob}
Let $\phi_1,\ldots,\phi_n\in\Reals$ and
let $\hvec$ be a sequence such that $\calT_{\hvec}\ne\emptyset$.
Then
\[
   \frac{1}{\card{\calT_{\hvec}}} \sum_{T\in \calT_{\hvec}}
      \sum_{jk\in E(T)} \phi_j\phi_k
   = \frac{1}{n-2}\,\biggl(\, \biggl(\,
    \sum_{k=1}^n\phi_k\biggr)\biggl(\,\sum_{j=1}^n (h_j{-}1)\phi_j\biggr)
                     - \biggl(\sum_{j=1}^n (h_j{-}1)\phi_j^2\biggr)\biggr)
\]
and
\[
    \frac{1}{\card{\calT_{\hvec}}} \sum_{T\in \calT_{\hvec}} \exp\biggl(\, 
        - \sum_{jk\in E(T)} \phi_j\phi_k\biggr)
    = \exp\biggl( K - \frac{1}{\card{\calT_{\hvec}}} \sum_{T\in \calT_{\hvec}}
         \sum_{jk\in E(T)} \phi_j\phi_k  \biggr)
\]
for some $K$ with
$\abs{K}\leq \dfrac18 n (\max_j\,\abs{\phi_j}-\min_j\,\abs{\phi_j})^4$.
\end{lemma}

Define
\begin{equation}\label{def_eps2}
 \eps_2(\eta) = \min\bigl\{\eps_5(\eta),\tfrac18\eta\bigr\},
 \end{equation}
where $\eps_5(\eta)$ is provided by Theorem~\ref{treeprob}.

\begin{lemma}\label{l:good_trees}
      Let $\eta \in(0,\dfrac12)$ be constant. If assumption \eqref{curlyA} holds with
      $\eps \in(0,\eps_2(\eta)]$ then the expected number of 
      good spanning trees is 
\[
 \sum_{T\in\calT_{\good}} \Prob(\dvec,T) = n^{n-2}\, \lambda^{n-1}
  \exp\left(-\frac{1-\lambda}{2\lambda} - \frac{R}{2\lambda^2 n} + O(n^{-1/2+\eta})\right).
\]
\end{lemma}
\begin{proof}
For a good tree $T$, we can apply Theorem~\ref{treeprob}  to estimate $\Prob(\dvec,T)$.
The function $f(\dvec,\hvec)$ in Theorem~\ref{treeprob}
depends only on $\hvec$, not on the tree $T$ itself.
We can ``average out'' the contribution of $g(\dvec,T)$,
which depends on the structure of~$T$,
using the function $\bar g(\dvec,\hvec)$ defined by
\[
   e^{\bar g(\dvec,\hvec)} = \frac{1}{\card{\calT_{\hvec}}} \,
    \sum_{T\in\calT_{\hvec}}e^{g(\dvec,T)}.
\]
Then we can write
\begin{align*}
 \sum_{T\in\calT_{\hvec}} \Prob(\dvec,T) 
  &= \lambda^{n-1}\, \exp\(f(\dvec,\hvec) + O(n^{-1/2+\eta})\)\, 
     \sum_{T\in\calT_{\hvec}} e^{g(\dvec,T)}\\ 
  &= \lambda^{n-1}\, \card{\calT_{\hvec}}\, \exp\(f(\dvec,\hvec) 
      + \bar g(\dvec,\hvec) + O(n^{-1/2+\eta})\).
\end{align*}
Define
\[ \phi_j = \frac{d_j-d-h_j+\lambda h_j}{n\sqrt{\lambda(1-\lambda)}}\]
for $j=1,\ldots, n$.
Using $\sum_{j=1}^n (d_j-d)=0$, $\norm{\hvec}\leq n^{3\eps}$
and $\Abs{\sum_{j=1}^n h_ja_j}=O(n\max_j \abs{a_j})$ for any $a_1,\ldots,a_n$,
the first part of Lemma~\ref{treeedgeprob} gives
\[
   \frac{1}{\card{\calT_{\hvec}}} \sum_{T\in \calT_{\hvec}}
      \sum_{jk\in E(T)} \phi_j\phi_k = O(n^{-1/2+2\eps}).
\]
Combining this with the second part of Lemma~\ref{treeedgeprob}, we find that
$\bar g(\dvec,\hvec) = O(n^{-1/2+2\eps})$.
The assumption  $\norm{\hvec}\leq n^{3\eps}$ and the fact that
$\mu_1=2-\frac2n$ for a tree imply that
\begin{align*}
    \frac{1-\lambda}{4\lambda}(\mu_1^2+2\mu_1)
    - \frac{(1-\lambda^2)}{6\lambda^2n} \mu_3
    &+ \frac{1}{2\lambda^2n^2}\sum_{j=1}^n (d_j-d)h_j^2 \\
    &{}= \frac{2(1-\lambda)}{\lambda} + O(n^{-1/2+8\eps}+n^{-1+10\eps}).
\end{align*}

Applying Theorem~\ref{treeprob}, we have, since $\eps\leq\frac18\eta$,
\begin{equation}
\label{goodsum}
  \sum_{T\in\calT_\good} \Prob(\dvec,T)
    = \(1+O(n^{-1/2+\eta})\) \lambda^{n-1} \sum_{\hvec\in D_\good}
        \binom{n-2}{h_1{-}1,\ldots,h_n{-}1} \,e^{f^*(\dvec,\hvec)},
\end{equation}
where
\[
f^*(\dvec,\hvec) = \frac{2(1-\lambda)}{\lambda}
 - \frac{1-\lambda}{2\lambda}\mu_2
- \frac{1}{2\lambda^2 n^2}\sum_{j=1}^n (d_j-d)^2 h_j 
  + \frac{1}{\lambda n} \sum_{j=1}^n (d_j-d)h_j.
\]

We now rewrite (\ref{goodsum}) using the random variable $\Z$ defined
in Lemma~\ref{distributions}, by extending the
sum over $D_\good$ to all of $D$, as follows: 
\begin{equation}
    \sum_{T\in\calT_\good} \Prob(\dvec,T)
    = n^{n-2} \lambda^{n-1} e^{O(n^{-1/2+\eta})} \left( \E e^{f^*(\dvec,\Z)}
       - \myE {\One_{D_\bad} e^{f^*(\dvec,\Z)}} \right).
\label{eq1}
\end{equation}

Applying the estimates from Lemma~\ref{distributions}(iii) 
to $f^*(\dvec,\Z)$,   we have
\begin{equation}\label{f*stats}
  \begin{split}
    \E f^*(\dvec,\Z) &= -\frac{1-\lambda}{2\lambda}
           - \frac{R}{\lambda^2 n} + O(n^{-1/2+\eta}), \\
    \Var f^*(\dvec,\Z) &= \frac{R}{\lambda^2 n} + O(n^{-1/2+\eta}).
  \end{split}
\end{equation}
Observe that by Lemma~\ref{distributions}(i), $f^*(\dvec,\Z)$ can be written as a function
of a multinomial distribution:
\[ f^*(\dvec,\Z) = \Tilde{f}(X_1-1,\ldots, X_n-1).\]
We will apply Theorem~\ref{Theorem_subsets}(iii) to the function $\Tilde{f}$.
Recalling (\ref{alphamax-vector}) and (\ref{Deltamax-vector}), we calculate that
$\alphamax=O(n^{-1/2+2\eps})$ and $\Deltamax=0$.
Therefore, using~\eqref{f*stats},
\begin{equation}
\label{eq2}
   \E e^{f^*(\dvec,\Z)} = \exp\biggl(
       -\frac{1-\lambda}{2\lambda}
       - \frac{R}{2\lambda^2 n} + O(n^{-1/2+\eta})
    \biggr).
\end{equation}

Next we bound $\myE {\One_{D_\bad} e^{f^*(\dvec,\Z)}}$.
From the definition of $f$ we have
$f^*(\zvec)\leq \hat f(\zvec)$ for all $\zvec\in\{1,2,\ldots\,\}^n$, where
$\hat f(\zvec) = \frac{2}{\lambda} +\frac{1}{\lambda n} \sum_{j=1}^n(d_j-d)z_j$.
For $\sigma\in S_n$, define 
\[
  \hat f_\sigma(\zvec) = \frac2\lambda
  + \frac{1}{\lambda n} \sum_{j=1}^n(d_{\sigma_j}-d)z_j.
\]
We now apply Lemma~\ref{ranper}
to estimate $\frac1{n!}\sum_{\sigma\in S_n}e^{\hat f_\sigma(\zvec)}$ for
$\zvec\in\{1,2,\ldots\,\}^n$. Defining
\[ u_j=\frac{d_j-d}{\lambda n} \quad \text{ and } \quad v_j=z_j\]
for $j=1,\ldots n$,
we find with respect to a uniformly random
permutation $\sigma\in S_n$ that $\hat f_\sigma(\zvec)$ has
expectation $\frac2\lambda$ and variance at most
$\frac{R}{\lambda^2 n(n-1)}\sum_{j=1}^n z_j^2$.
The parameter $\alpha$ required by
Lemma~\ref{ranper} satisfies $\alpha = O(n^{-1/2+4\eps}/\lambda)$.
Consequently, by Lemma~\ref{ranper}(i),
\begin{equation}
\label{consequently}
 \frac1{n!}\sum_{\sigma\in S_n}e^{\hat f_\sigma(\zvec)}
 =O(e^{2/\lambda})\,  e^{ C \sum_{j=1}^n z_j^2}
\end{equation} 
where $C = \dfrac{R}{2\lambda^2 n(n-1)}$.

Since $D_\good$ is invariant under permutations of the
components, $\E e^{\hat f(\Z)}$ is a symmetric function of
$d_1,\ldots,d_n$.  Therefore,
\begin{align*}
\myE {\One_{D_\bad} e^{f(\Z)}} \leq \myE {\One_{D_\bad} e^{\hat f(\Z)}}
&= \MYE{\One_{D_\bad} \frac1{n!}\sum_{\sigma\in S_n}e^{\hat f_\sigma(\Z)}}\\
&= O(e^{2/\lambda}) \MyE{\One_{D_\bad} e^{C\sum_{j=1}^n Z_j^2}}.
\end{align*}
Next, note that $Y_1 + \cdots + Y_n$ has a Poisson
distribution with mean $n$, and hence by Stirling's approximation, 
\[ \Prob(Y_1+\cdots+Y_n=n-2)= \frac{e^{-n}\, n^{n-2}}{(n-2)!} = \Theta(n^{-1/2}).
\]
Applying Lemma~\ref{distributions}(ii), we obtain
\begin{equation}
\label{overcount} \Prob(X_1 = y_1+1,\ldots, X_n = y_n+1) 
= O(n^{1/2})\, \Prob(Y_1=y_1,\ldots,Y_n=y_n).
\end{equation}
Therefore,
\[
  \myE {\One_{D_\bad} e^{f(\Z)}} 
  = O(e^{2/\lambda}n^{1/2})
 \sum_{y_1,\ldots,y_n} \Prob\(\Y=(y_1,\ldots,y_n)\)\, 
      e^{C\sum_{j=1}^n \min\{y_j+1,\lfloor n^{3\eps}\rfloor\}^2}, 
\]
where the sum is restricted to sequences $(y_1,\ldots, y_n)$ of
nonnegative integers such that
$(y_1+1,\ldots, y_n+1)\not\in D_{\text{good}}$.
Recalling that the components of $\Y$ are independent, we can separate the
sum and use the union bound on the constraint.
This gives 
\[ \myE{ \One_{D_\bad} e^{f(\Z)}} \leq O(e^{2/\lambda}n^{3/2})
  (\varSigma_1+\varSigma_2)^{n-1} \varSigma_2,\]
 where
\[
     \varSigma_1 = \sum_{y=0}^{\lfloor n^{3\eps}\rfloor} \frac{e^{-1}}{y!}\, e^{C(y+1)^2}
     \quad\text{and}\quad
     \varSigma_2 = \sum_{y=\lfloor n^{3\eps}\rfloor+1}^\infty \frac{e^{-1}}{y!}\, e^{Cn^{6\eps}}.
\]
Since $C=O(n^{-1+2\eps}/\lambda^2)$ and $\sum_{j=0}^\infty \frac1{j!}(j+1)^2=5e$,
we conclude that
\begin{align}
  \varSigma_1 &= \sum_{y=0}^{\lfloor n^{3\eps}\rfloor -1} \frac{e^{-1}}{y!}\,
     \(1 + C(y{+}1)^2 + O(n^{-2+17\eps})\)
    = 1 +  5C + O(n^{-2+17\eps})  \label{Sigma1}, \\
  \varSigma_2 &= O(e^{-n^{3\eps}}).\notag
\end{align}
Therefore
 \begin{equation}
\label{eq3}
\myE{ \One_{D_\bad} e^{f(\Z)}} = O(n^{3/2}) 
   \exp\( 2/\lambda -n^{3\eps}  + O(n^{2\eps}/\lambda^2) \) = 
                    O\(e^{- n^{3\eps}/2}\).
\end{equation}
Combining (\ref{eq1}), (\ref{eq2}) and (\ref{eq3}) completes the proof.
\end{proof}

\subsection{The expected number of bad spanning trees}\label{s:bad-trees}

To complete the proof of Theorem~\ref{treecount},
it remains for us to bound $\sum_{T\in\calT_\bad} \Prob(\dvec,T)$.  Note that we
cannot use Theorem~\ref{treeprob} directly since $T$ fails the
required degree bound.  
However, we can choose a
subgraph $F\subseteq T$ to which Theorem~\ref{treeprob} applies and use the fact that $\Prob(\dvec,T)\leq \Prob(\dvec,F)$.

\begin{lemma}\label{l:bad_trees}
      Let $\eta \in(0,\frac12)$ be constant. If assumption \eqref{curlyA} holds with
      $\eps \in(0,\eps_2(\eta)]$, where $\eps_2(\eta)$ is defined in \eqref{def_eps2},
       then the expected number of bad spanning trees is 
 \[
 \sum_{T\in\calT_{\bad}} \Prob(\dvec,T) = n^{n-2}\, \lambda^{n-1}\,
   O\(e^{-n^{3\eps}/2}\).
\]
\end{lemma}
\begin{proof}
Let $T$ be a bad tree.
Define $F(T)$ to be the set of all subgraphs of $T$ that have maximum
degree at most $n^{3\eps}$ and at least
$n-1-\sum_{j=1}^n \max\{0,x_j-n^{3\eps}\}$ edges.
Since one such subgraph is obtained by deleting $\max\{0,x_j-\lfloor n^{3\eps}\rfloor \}$
arbitrary edges incident with each vertex~$j$, we have $F(T)\ne\emptyset$.
We also have that for any permutation $\sigma$ of the vertices,
$F(\sigma(T))=\sigma(F(T))$.
Since $g(\dvec,F')=O(n^{2\eps}/\lambda(1-\lambda))$
for all $F'\in F(T)$, using Theorem~\ref{treeprob} we can write
\[
   \sum_{T\in\calT_\bad} \Prob(\dvec,T) \leq
    e^{\frac{1}{\lambda(1-\lambda)} O(n^{2\eps})}\,
      \sum_{T\in\calT_\bad} 
    \lambda^{n-1-\sum_{j=1}^n \max\{0,x_j-n^{3\eps}\}}
    \frac{1}{\card{F(T)}}\sum_{F'\in F(T)} e^{\hat f(\zvec(F'))},
\]
where $\zvec(F')$ is the degree sequence of $F'$ and
$\hat f$ is defined as before.
The expression on the right is a symmetric function of $\dvec$,
so we can average it over all permutations of the elements
of~$\dvec$.  The same calculations that led to (\ref{consequently})
show that
\[
   \frac1{n!} \sum_{\sigma\in S_n} 
    e^{\hat f_\sigma(\zvec(F'))} = 
  O(e^{2/\lambda}) e^{C\sum_{j=1}^n \min\{x_j,n^{3\eps}\}^2}.
\]
Therefore
\begin{align*}
   & \sum_{T\in\calT_\bad} \Prob(\dvec,T) \\
 &\leq
    e^{\frac{1}{\lambda(1-\lambda)} O(n^{2\eps})}\,
 \lambda^{n-1}\, \sum_{T\in\calT_\bad} \lambda^{-\sum_{j=1}^n \max\{0,x_j-n^{3\eps}\}}
   e^{C\sum_{j=1}^n \min\{x_j,n^{3\eps}\}^2}\\
 &\leq 
    e^{\frac{1}{\lambda(1-\lambda)} O(n^{2\eps})}\,
 \,\lambda^{n-1} n^{n-2} \\
 & \hspace*{2cm} 
  \times\sum_{y_1,\ldots,y_n} \Prob\(\Y=(y_1,\ldots,y_n)\) \,
 \lambda^{-\sum_{j=1}^n \max\{0,y_j+1-n^{3\eps}\}}\,
 e^{C\sum_{j=1}^n \min\{y_j+1,n^{3\eps}\}^2},
\end{align*}
using (\ref{overcount}).
As before, the sum is restricted to those sequences $(y_1,\ldots, y_n)$ of 
nonnegative integers such that
$(y_1+1,\ldots, y_n+1)\not\in D_{\text{good}}$.

Separating the sum and applying the union bound, we have
\[ 
  \sum_{y_1,\ldots,y_n} \Prob\(\Y=(y_1,\ldots,y_n)\) \,
 \lambda^{-\sum_{j=1}^n \max\{0,y_j+1-n^{3\eps}\}}\,
 e^{C\sum_{j=1}^n \min\{y_j+1,n^{3\eps}\}^2}
\leq 
n\, (\varSigma_1+\varSigma'_2)^{n-1} \varSigma'_2, 
\]
where  $\varSigma_1$ is defined earlier and
\begin{align*}
  \varSigma'_2 &= \sum_{y=\lfloor n^{3\eps}\rfloor}^\infty \frac{e^{-1}\lambda^{-(y-n^{3\eps})}}{y!}\, e^{Cn^{6\eps}}
    = O(e^{-n^{3\eps}}).
\end{align*}
Therefore, using (\ref{Sigma1}), 
\begin{align*}
   \sum_{T\in\calT_\bad} \Prob(\dvec,T) 
 &\leq
    \lambda^{n-1}\, n^{n-2}\,
    \exp\Bigl( - n^{3\eps} + \lambda^{-2} n^{2\eps}  + \dfrac{1}{\lambda(1-\lambda)} O(n^{2\eps})\Bigr)
   \notag \\[-1ex]
   &= \lambda^{n-1}\, n^{n-2}\, O(e^{-n^{3\eps}/2}). \qedhere
\end{align*}
\end{proof}

\section{Counting induced subgraphs}\label{s:induced}

Recall the parameters defined in (\ref{R-lambda-delta}). 
In this section, our starting point is the following result adapted from  McKay~\cite[Theorem~2.4]{ranx}
in the same way as Theorem \ref{treeprob} was adapted from  \cite[Theorem~2.1]{ranx}.
We require notation that generalizes \eqref{omega}:
 \[
 \omega_{s,t} = \sum_{j=1}^r \,(d_{j}-d)^s (h_j - \lambda(r-1))^t
  \qquad 
    \text{for $s,t \geq 0$}.
\]

\begin{thm}\label{induced-prob}
Let $\eta\in(0,\frac12)$ be constant. Then there is
a constant $\eps_6(\eta) > 0$ such that the following
holds for every fixed $\eps \in (0,\eps_6(\eta)]$.
Let $\dvec$ be a degree sequence which satisfies~\eqref{curlyA}.
Suppose that $H^{[r]}$ is a graph on the vertex set $\{ 1,\ldots, r\}$ with degree
sequence $\hvec^{[r]} = (h_1,\ldots, h_r)$ such that
$r \leq n^{1/2 + \eps}$.
Then the probability that $G\sim \Gd$ has $H^{[r]}$ as an induced subgraph is
\begin{align*}
     \lambda^m \, &(1-\lambda)^{\binom{r}{2}-m}\,\\
 & \times \exp\biggl(\frac{2\omega_{1,1} - \omega_{0,2}}{2\lambda(1-\lambda) n} + \frac{r^2}{2n}
   + \frac{(1-2\lambda)\omega_{0,1}}{2\lambda(1-\lambda) n} +
    \frac{4\omega_{1,0}\omega_{0,1} - \omega_{0,1}^2 - 2\omega_{1,0}^2}{4\lambda(1-\lambda) n^2} \\
   & \hspace*{2cm} {} + \frac{r(2\omega_{1,1} - \omega_{2,0} - \omega_{0,2})}{2\lambda(1-\lambda) n^2}
   - \frac{(1-2\lambda)(\omega_{0,3} + 3\omega_{2,1} - 3\omega_{1,2})}{6\lambda^2(1-\lambda)^2 n^2}
   + O(n^{-1/2+\eta}) \biggr).
\end{align*}
\end{thm}

Now, for a given permutation $\sigma\in S_n$, let
\[ \omega_{s,t}(\sigma) = \sum_{j=1}^r \,(d_{\sigma_j}-d)^s\, (h_j - \lambda(r-1))^t.
\]
Note that $\omega_{0,t}(\sigma)$ is independent of $\sigma$ and equals $\omega_{t}$ from (\ref{omega}).
Let $f_{\text{ind}}:S_n\rightarrow \Reals $ be defined as
\begin{align*} f_{\text{ind}}(\sigma) = 
 &\frac{2\omega_{1,1}(\sigma) - \omega_{2}}{2\lambda(1-\lambda) n} + \frac{r^2}{2n}
   + \frac{(1-2\lambda)\omega_{1}}{2\lambda(1-\lambda) n} +
    \frac{4\omega_{1,0}(\sigma)\omega_{1} - \omega_{1}^2 - 2\omega_{1,0}(\sigma)^2}{4\lambda(1-\lambda) n^2} \\
   & \quad  {} + \frac{r(2\omega_{1,1}(\sigma) - \omega_{2,0}(\sigma) - \omega_{2})}{2\lambda(1-\lambda) n^2}
   - \frac{(1-2\lambda)(\omega_{3} + 3\omega_{2,1}(\sigma) - 3\omega_{1,2}(\sigma))}{6\lambda^2(1-\lambda)^2 n^2}.
\end{align*}
Observe that the considerations of Remark~\ref{uniformremark} apply to
Theorem~\ref{induced-prob}.
Thus, we find (under the assumptions of Theorem \ref{induced-prob}) that the expected number of induced copies of
$H^{[r]}$ in a
uniformly random graph with degree sequence $\dvec$ is
\begin{equation}\label{ind_start-point}
\(1+O(n^{-1/2+\eta})\) \frac{r!}{|\Aut(H^{[r]})|} \binom{n}{r} \lambda^m\, (1-\lambda)^{\binom{r}{2}-m}\,    \mYE{e^{f_{\text{ind}}(\X)}} ,
\end{equation}
where  the expectation is taken with respect to a uniformly random element $\X$ of
$S_n$.

\bigskip

In the proof of Theorem \ref{inducedsubcount} we will use the following bounds
given by the power mean inequality:
\begin{equation}\label{bounds_PM}
 \begin{aligned}
 \frac{\delta}{\lambda(1-\lambda) n} \sum_{j=1}^r |h_j - \lambda(r-1)| \leq 
    r^{2/3} \biggl(\frac{\delta^3}{\lambda^3(1-\lambda)^3 n^3}\sum_{j=1}^r |h_j - \lambda(r-1)|^3 \biggr)^{\!\!1/3},\\
    \frac{\delta^2}{\lambda^2(1-\lambda)^2 n^2} \sum_{j=1}^r |h_j - \lambda(r-1)|^2\leq 
    r^{1/3} \biggl(\frac{\delta^3}{\lambda^3(1-\lambda)^3 n^3}\sum_{j=1}^r |h_j - \lambda(r-1)|^3 \biggr)^{\!\!2/3}.
   \end{aligned}
\end{equation}


Before proving Theorem~\ref{inducedsubcount}, we apply the results of Section~\ref{s:moments} to obtain the following expressions. Define 
\begin{equation}\label{def_eps3}
\eps_3(\eta)=\min\bigl\{\eps_6(\eta),\tfrac18\eta\bigr\}.
\end{equation}

\begin{lemma}\label{inducedallvar}
Let $\eta\in(0,\frac12)$ be constant. 
If assumptions \eqref{curlyA} and \eqref{induced-assumption}
hold with  $\eps \in (0,\eps_3(\eta)]$, then
\begin{align*}
     \E f_{\rm ind}(\X) 
&= -\frac{\omega_{2}}{2\lambda(1-\lambda) n} + \frac{r^2}{2n} + \frac{(1-2\lambda) \omega_{1}}{2\lambda(1-\lambda)
n}
  - \frac{\omega_{1}^2}{4\lambda(1-\lambda) n^2} - \frac{r^2 R}{2\lambda(1-\lambda)n^2}\\
& \qquad \qquad {} - \frac{r\, \omega_{2} }{2\lambda(1-\lambda) n^2} - \frac{(1-2\lambda) \omega_{3}}{6\lambda^2
(1-\lambda)^2 n^2} - \frac{(1-2\lambda) R \, \omega_{1}}{2\lambda^2(1-\lambda)^2 n^2} + O(n^{-1/2+\eta}),\\
\Var f_{\rm ind} (\X) &= \frac{R\, \omega_{2}}{\lambda^2(1-\lambda)^2 n^2} -\frac{r\, \omega_{1} \sum_{j=1}^n
(d_j-d)^3}{\lambda^2(1-\lambda)^2 n^4} + O(n^{-1/2+\eta}).
\end{align*}
\end{lemma}

\begin{proof}
 We will often employ the bounds $R \leq\delta^2$ and $\abs{h_j - \lambda(r-1)} \leq r$. 
 
In order to easily apply Lemma~\ref{ranper}, we extend the sum
defining $\omega_{s,t}(\sigma)$ to $n$ terms by appending zeros:
\begin{equation}
\label{omegakl}
\omega_{s,t}(\sigma) = \sum_{j=1}^n u_j v_{\sigma_j}
\end{equation}
where for $j=1,\ldots n$,
\[
 u_j = \begin{cases} \;(h_j - \lambda(r-1))^t, & \text{if $j\leq r$},\\
          \;0, & \text{if $j\geq r+1$}, \end{cases} \qquad\text{and}\qquad
 v_j = (d_j - d)^s.
\]

When applying Lemma~\ref{ranper}, it will be convenient to use the identity
\[
   \sum_{j=1}^n (u(j)-\bar{u}) (u'(j)-\bar{u}') =  \sum_{j=1}^n u(j) u'(j) - n \bar{u} \bar{u}'.    
\]
If $u(j) = q_j^k$ and  $u'(j) = q_j^t$ for $j=1,\ldots, n$, for some sequence 
$(q_1,\ldots,q_n)\in \Reals^n$ and  $k,t \geq 0$,
then we can apply the power mean inequality to bound this expression by $ O(1) \sum_{j=1}^n |q_j|^{k+t}$.

By Lemma~\ref{ranper}(i),
\[
 \myE {\omega_{s,t}(\X)} = \frac{1}{n} \biggl(\,\sum_{j=1}^n (d_j-d)^s\biggr)
                  \biggl(\,\sum_{j=1}^r (h_j-\lambda(r-1))^{t}\biggr).
\]
This implies that $\myE{\omega_{1,t}(\X)} = 0$ for any $t\geq 0$, and that
\[ 
  \MYE{-\frac{r\,\omega_{2,0}(\X) }{2\lambda(1-\lambda) n^2}}
  = -\frac{r^2R}{2\lambda(1-\lambda) n^2},\qquad
 \MYE{-\frac{(1-2\lambda)\omega_{2,1}(\X)}{2\lambda^2(1-\lambda)^2 n^2}}
 = -\frac{(1-2\lambda) R\, \omega_{1}}{2\lambda^2 (1-\lambda)^2 n^2}.
\]
Finally, applying Lemma~\ref{ranper}(ii) shows that
\begin{align*} 
\myE{\omega_{1,0}(\X)^2} = \myVar{\omega_{1,0}(\X)}
   &= \frac{1}{n-1} \sum_{j=1}^n (d_j-d)^2 
    \left( r\left(1-\frac{r}{n}\right)^2 + (n-r)\left(\frac{r}{n}\right)^2\right)\\
 &= O(R r),
\end{align*}
again using the fact that $\myE{\omega_{1,0}(\X)}=0$ for the first equality.
Therefore 
\[ \MYE{ - \frac{\omega_{1,0}(\X)^2}{2\lambda(1-\lambda) n^2}}
  =  O\left(\frac{Rr}{\lambda (1-\lambda) n^2}\right) = O(n^{-1/2 + 4\eps}) = O(n^{-1/2+\eta}).
\]
Combining the above expressions and estimates leads to the
expression for
$\E f_{\rm ind}(\X)$. 

Now for the variance. From Lemma~\ref{ranper}(ii), we have
\begin{equation}\label{variance_11}
 \myVar{\omega_{1,1}(\X)}  = \frac{nR}{n-1}\, \left(\omega_{2} - \frac{\omega_{1}^2}{n}\right)
  = R\, \omega_{2} + O\left(\frac{R (\omega_{2} + \omega_{1}^2)}{n}\right).
\end{equation}
Combining \eqref{induced-assumption} and \eqref{bounds_PM} gives
\begin{align*}
\frac{ R\, \omega_{1}^2}{\lambda^2(1-\lambda)^2 n^3}
  &= O\left(\frac{r^{4/3} n^{-1/3+2\eta/3}}{n}\right) =O(n^{-1/2+\eta}),\\
\frac{ R\, \omega_{2}}{\lambda^2(1-\lambda)^2 n^3}
 &= O\left(\frac{r^{1/3} n^{-1/3+2\eta/3}}{n}\right)
= O(n^{-1/2+\eta}).
\end{align*}
Therefore, the first term of $f_{\rm ind}$ has variance
\[ 
  \MYVar{\frac{\omega_{1,1}(\X)}{\lambda(1-\lambda) n}}
  = \frac{R\omega_{2}}{\lambda^2(1-\lambda)^2 n^2} + O(n^{-1/2+\eta}).
\]
Also
\[ 
  \MYVar{\frac{r\, \omega_{1,1}(\X)}{\lambda(1-\lambda) n^2}}
  = O\left(\frac{r^2 R\, \omega_{2}}{\lambda^2(1-\lambda)^2\, n^4}\right) = 
   O\left(\frac{\delta^2 r^5}{\lambda^2(1-\lambda)^2\, n^4}\right)
  = O(n^{-1/2+\eta}).
\]
Using Lemma~\ref{ranper}(ii) we have the following rough bound:
\begin{equation}\label{rough} 
\myVar{\omega_{k,t}(\X)} = O(\delta^{2k}\, r^{2t + 1}).
\end{equation}
Hence
\begin{align*}
\MYVar{\frac{\omega_{1,0}(\X)\,\omega_{1}}{4\lambda(1-\lambda) n^2}}
  &= 
   O\left(\frac{r^5\, \delta^2}{\lambda^2 (1-\lambda)^2 n^4}\right) = O(n^{-1/2+\eta}),\\
\MYVar{\frac{r\, \omega_{2,0}(\X)}{\lambda(1-\lambda) n^2}}
  &= O\left(\frac{r^3 \delta^4}{\lambda^2 (1-\lambda)^2 n^4}\right) = O(n^{-1/2+\eta}),\\
\MYVar{\frac{\omega_{2,1}(\X)}{\lambda^2(1-\lambda)^2 n^2}}
  &= O\left(\frac{\delta^4 r^3}{\lambda^4(1-\lambda)^4 n^4}\right) = O(n^{-1/2+\eta}),\\
\MYVar{\frac{\omega_{1,2}(\X)}{\lambda^2(1-\lambda)^2 n^2}}
 &= O\left(\frac{\delta^2 r^5}{\lambda^4 (1-\lambda)^4 n^4}\right) = O(n^{-1/2+\eta}).
\end{align*}

The final variance that we must calculate is $\myVar{\omega_{1,0}(\X)^2}$.
We have
\begin{equation}\label{varw102}
  \myVar{\omega_{1,0}(\X)^2} \leq \myE {\omega_{1,0}(\X)^4}
  = \frac{1}{n!}\sum_{\sigma\in S_n} \Bigl(\sum_{j=1}^r (d_{\sigma_j}-d)\Bigr)^4
    - \frac{(r)_4}{(n)_4}\Bigl(\sum_{k=1}^n (d_k-d)\Bigr)^4.
\end{equation}
Note that the second term is  $0$ 
since $\sum_{k=1}^n (d_k-d)=0$.
Expanding  the right hand side of  \eqref{varw102} as
\[
    \MYE{\,\sum_{i_1,i_2,i_3,i_4=1}^n\!\!
       c(i_1,i_2,i_3,i_4) (d_{i_1}-d)(d_{i_2}-d)(d_{i_3}-d)(d_{i_4}-d)},
\]
we find that $c(i_1,i_2,i_3,i_4)=0$ if $i_1,i_2,i_3,i_4$ are distinct.
The  part $ \Bigl(\sum_{j=1}^r (d_{\sigma_j}-d)\Bigr)^4$ has $r^4-(r)_4=O(r^3)$ other terms
while the part $\frac{(r)_4}{(n)_4}\Bigl(\sum_{k=1}^n (d_k-d)\Bigr)^4$
 has $n^4-(n)_4=O(n^3)$ other terms.
Therefore, $\myVar {\omega_{1,0}(\X)^2}=O(r^3\delta^4+r^4\delta^4/n)
=O(r^3\delta^4)$, which proves that
\[
    \MYVar{\frac{\omega_{1,0}(\X)^2}{2\lambda(1-\lambda) n^2}}
  = O(n^{-1/2+\eta}).
\]
Hence we see that only the first term of $f_{\rm ind}(\X)$ has non-negligible
variance.  It follows that any covariance which does not involve the first term
will automatically fit within the $O(n^{-1/2+\eta})$ error term.
We now compute the remaining
covariances. Lemma~\ref{ranper}(ii) implies that
\[ 
 \mYCov{\omega_{j,k}(\X),\omega_{s,t}(\X)} = O(\delta^{j+s}\, r^{k+t+1}),
\]
which shows that
\[
\MYCov{\frac{\omega_{1,1}(\X)}{\lambda(1-\lambda)n}, 
  \frac{\omega_{1,0}(\X)}{\lambda(1-\lambda)n^3}}
    = O\left(\frac{\delta^2 r^2}{\lambda^2(1-\lambda)^2 n^3}\right) = O(n^{-1/2+\eta}).
\]
Using assumption  \eqref{induced-assumption} and bounds \eqref{bounds_PM}, \eqref{variance_11}, we find that
\begin{align*}
\MYCov{\frac{\omega_{1,1}(\X)}{\lambda(1-\lambda)n},
  \frac{r\, \omega_{1,1}(\X)}{\lambda(1-\lambda)n^2}}
 &= \frac{r}{\lambda^2 (1-\lambda)^2 n^3}\, \myVar{\omega_{1,1}(\X)} \\
  &= O\left(\frac{\delta^2 r\omega_{2}}{\lambda^2(1-\lambda)^2 n^3}\right)
  = O \left( \frac{r^{4/3} n^{-1/3+2\eta/3}}{n} \right) = O(n^{-1/2+\eta}).
\end{align*}
Applying Lemma~\ref{ranper}(ii) and using the same kind of argument, we find that
\begin{align*} 
  \MYCov{\frac{\omega_{1,1}(\X)}{\lambda(1-\lambda)n},
  \frac{\omega_{2,1}(\X)}{\lambda^2(1-\lambda)^2 n^2}}
  & = O\left(\frac{\delta^3 \omega_{2}}{\lambda^3(1-\lambda)^3 n^3}\right) \\
  &{\kern3em}= O\left(\frac{\delta r^{1/3} n^{-1/3+2\eta/3}}{\lambda(1-\lambda) n }\right) = O(n^{-1/2+\eta}).
\\
  \MYCov{\frac{\omega_{1,1}(\X)}{\lambda(1-\lambda)n},
  \frac{\omega_{1,2}(\X)}{\lambda^2(1-\lambda)^2 n^2}}
  &= O\left(\frac{\delta^2 \sum_{j=1}^r |h_j -\lambda(r-1)|^3}{\lambda^3(1-\lambda)^3 n^3}\right)
= O(n^{-1/2+\eta}).
\end{align*}

The following term will contribute to the answer, so we calculate it precisely.
Writing $\omega_{1,1}(\sigma))=\sum_{j=1}^n u_jv_{\sigma_j}$ using (\ref{omegakl}) with $s=t=1$, 
we have $\bar{u} = \omega_{1}/n$ and
$\bar{v} = 0$.  Similarly, write $\omega_{2,0}(\sigma) = \sum_{j=1}^n u'_j v'_{\sigma_j}$ using (\ref{omegakl})
with $s=2$ and $t=0$, giving $\bar{u}' = r/n$ and $\bar{v}'=R$.  Applying Lemma~\ref{ranper}(ii) gives
\begin{align} 
 & \MYCov{\frac{\omega_{1,1}(\X)}{\lambda(1-\lambda)n}, 
  \frac{-r \, \omega_{2,0}(\X)}{2\lambda(1-\lambda)n^2}} \notag \\
    &=  -\frac{r}{2\lambda^2(1-\lambda)^2 n^3(n-1)}\,
    \sum_{j=1}^n (d_j-d)\((d_j-d)^2 - R\)\notag \\
         & \hspace*{3cm} \qquad \quad \times 
  \biggl(\,\sum_{j=1}^r \left(h_j-\lambda(r-1) - \dfrac{\omega_{1}}{n}\right)\left(1-\dfrac{r}{n}\right) + 
          \frac{\omega_{1}\, r(n-r)}{n^2}\biggr)\notag \\
    &=  -\frac{r\,  \omega_{1} \sum_{j=1}^n (d_j-d)^3}{2\lambda^2(1-\lambda)^2 n^4} + O(n^{-1/2+\eta}). \label{cov}
\end{align}

Finally, we need
\begin{equation}\label{needthis2}
 \MYCov{\frac{\omega_{1,1}(\X)}{\lambda(1-\lambda)n},\frac{\omega_{1,0}(\X)^2}{\lambda(1-\lambda) n^2}}
 = \frac{1}{\lambda^2(1-\lambda)^2 n^3}\,  \sum_{j=1}^r \sum_{k=1}^r
   \mYCov{\hat E_{jk}(\X),\omega_{1,1}(\X)},
\end{equation}
where $\hat E_{jk}(\sigma) = (d_{\sigma_j}-d)(d_{\sigma_k}-d)$.
If $j\neq k$ then each covariance in the 
sum on the right hand side matches the setting of Lemma~\ref{ranper}(v)
with $u_j=0$ for all $j$, and $v_j = d_j-d$. Note $\bar{u}=\bar{v}=0$.
Write $\Psi'(\X) = \omega_{1,1}(\X)$
using (\ref{omegakl}) with $s=t=1$, giving $\bar{u}' = \omega_{1}/n$ and $\bar{v}'=0$.
By Lemma~\ref{ranper}(iv), if $j\neq k$ then
\[ 
  \myCov{\hat E_{jk}(\X),\omega_{1,1}(\X)} = O\left(\frac{\delta^3 r }{n}\right)
\]
since $u(k) + \bar{v} = u(j) + \bar{v} = 0$ for all $j,k$.  
Therefore the terms in~\eqref{needthis2} with $j\ne k$ contribute
\[
 2\,\binom{r}{2}\, O\left(\frac{\delta^3 r}{\lambda^2 n^4}\right) = O(n^{-1/2+\eta}).
\]
The terms in~\eqref{needthis2} with $j=k$ contribute
\[ \frac{1}{\lambda^2(1-\lambda)^2 n^3}\, \myCov{\omega_{1,1}(\X),\, \omega_{2,0}(\X)}
  = O\left(\frac{\delta^3 \omega_{1} }{\lambda^2(1-\lambda)^2  n^3}\right) =  O(n^{-1/2+\eta}),
\]
using the earlier expression for this covariance \eqref{cov} and the bound $\omega_{1}=O(r^2)$. Thus we see that~\eqref{needthis2}
does not contribute significantly.

Combining all the estimates above (and multiplying (\ref{cov}) by 2) gives the stated
expression for the
variance of $f_{\rm ind}(\X)$.  This completes the proof.
\end{proof}

We can now prove our main result about induced subgraphs.

\begin{proof}[Proof of Theorem~\ref{inducedsubcount}]
Define $\eps_3(\eta)$ as in \eqref{def_eps3}.
We will apply Theorem~\ref{RanPermThm} to estimate (\ref{ind_start-point}). The expected value and variance of $f_{\rm ind}$ are given in
Lemma~\ref{inducedallvar}.  It remains to prove that 
\[ \sum_{j=1}^{n-1}\, \(\nfrac{1}{6} \alpha_j^3 + \nfrac{1}{3} \alpha_j \beta_j + \nfrac{5}{8} \alpha_j^4 + \nfrac{5}{8} \beta_j^2\)
  = O(n^{-1/2+\eta}),
\]
where $\alpha_{j} = \alpha_j[f_{\rm ind},S_n]$,  $\beta_j = \sum_{k=j+1}^{n-1} \alpha_k\varDelta_{jk}$ and $\varDelta_{jk} = \varDelta_{jk}[f_{\rm ind},S_n]$.

Without loss of generality, we can assume that
\[
   |h_1 - \lambda(r-1)| \geq |h_2 - \lambda(r-1)| \geq \cdots \geq |h_r - \lambda(r-1)|.
\]

For any $s,t$, and $1\leq j<a\leq n$, the function $\omega_{s,t}$ satisfies
\[
  \|D^{(ja)} \omega_{s,t}\| = 
  \begin{cases}
	   \,O( \delta^s |h_j - \lambda(r-1)|^t ), & \text{ for  } j \leq r;\\
	   \,0, &  \text{otherwise}.
	\end{cases}
\]
Also, we have 
\[
  \|D^{(ja)} \omega_{1,0}^2\|  \leq 2 \|\omega_{1,0}\|\|D^{(ja)} \omega_{1,0}\| =
	\begin{cases}
	   \,O( \delta^2 r ), & \text{ for  } j \leq r;\\
	   \,0, &  \text{otherwise}.
	\end{cases}
\]
Let $\alpha_j = \alpha_j[f_{\rm ind},S_n]$. 
Observe that $\alpha_j = 0$ for $j>r$.  By our assumptions, we have  
\[
   r,\delta, h_j = O(n^{1/2 + \eps}), \ \ \ \
	\omega_{1} = 2m - \lambda \binom{r}{2} =O(n^{1+2\eps}).
\]
Thus, using the above bounds, we find that $\alpha_j = O(\gamma_j)$ for $1\leq j\leq r$, where
 \[
    \gamma_j 
     =     
	   \frac{\delta |h_j - \lambda(r-1)|}{\lambda(1-\lambda) n} +  n^{-1/2+4\eps}.  
\]
Note that for any $s,t$, and
distinct $1\leq j,k,a,b\leq n$ with $j<a$ and $j<k<b$, we have 
\begin{align*}
  \|D^{(k\, b)}D^{(j\, a)} \omega_{s,t}\| &= 0,
 \\
  \|D^{(k\, b)} D^{(j\, a)} \omega_{1,0}^2\|  &=  
	\begin{cases}
	   \,O( \delta^2  ), & \text{ for  } k \leq r;\\
	   \,0, &  \text{otherwise}.
	\end{cases}
\end{align*}
Let $\varDelta_{jk}=  \varDelta_{jk} [f_{\rm ind},S_n]$. 
We have that $\varDelta_{jk} = 0$ for $k> r$.
Observe also that
\begin{align*}
   \|D^{(k\, a)} D^{(j\, a)} f_{\rm ind}\| &\leq 2\|D^{(k\, a)} f_{\rm ind}\| = O(n^{3\eps}), \\ 
   \|D^{(k\, b)} D^{(j\, k)} f_{\rm ind}\| &\leq 2\|D^{(j\, k)} f_{\rm ind}\| = O(n^{3\eps}).
\end{align*}
 Thus, using the bounds above, we find that $ \varDelta_{jk}  =  O(n^{-1+3\eps})$  for $1 \leq j <k\leq r$.
	
    Using the inequality  $(|x|+|y|)^3 \leq 4(|x|^3 + |y|^3)$ for each term of the sum, we find that
    \begin{align*}
      \sum_{j=1}^{n-1} \alpha_j^3  &= O(1) \sum_{j=1}^{r} \gamma_j^3= 
         \sum_{j=1}^r  O\biggl( \biggl( \frac{\delta |h_j - \lambda(r-1)|}{\lambda(1-\lambda) n} 
         + n^{-1/2+4\eps}\biggr)^{\!3}\,\biggr) \\ &=
O\biggl(   \frac{\delta^3 \sum_{j=1}^r |h_j - \lambda(r-1)|^3}{\lambda^3(1-\lambda)^3 n^3} 
  +rn^{-3/2+12\eps}\biggr) = O(n^{-1/2+\eta})
\end{align*}
   by \eqref{induced-assumption} and the bound $\eps\leq \frac18\eta$.  
  Observe that  $\beta_j = 0$ for $j>r$ and
  \[
    \beta_j  =  O(r \alpha_j n^{-1+3\eps}) =  O(\gamma_j\, n^{-1/2+4\eps})  \ \ \ \text{ for } j\leq r.
  \]
  Using the power mean inequality, we bound
 \begin{align*}
  	\sum_{j=1}^n \alpha_j \beta_j &= O(n^{-1/2+4\eps}) \sum_{j=1}^{r} \gamma_j^2 
  	 \leq   O(n^{-1/2+4\eps}) r^{1/3} \biggl(\sum_{j=1}^{r} \gamma_j^3\biggr)^{\!\!2/3}\\
  	  &=O(n^{-1/3+4\eps + \eps/3})\biggl(\sum_{j=1}^{r} \gamma_j^3\biggr)^{\!\!2/3}
	     = O(n^{-1/2+\eta})
  \end{align*}
 as before.
   The two remaining terms also have negligible contribution:
  \begin{align*} 
  \sum_{j=1}^{n-1} \alpha_j^4 &\leq   \sum_{j=1}^{r} \gamma_j^4 = O(1)\sum_{j=1}^{r} \gamma_j^3 = O(n^{-1/2+\eta}), \\
    \sum_{j=1}^{n-1} \beta_j^2  &=  O(n^{-1+8\eps})\sum_{j=1}^{r} \gamma_j^2 = O(n^{-1/2+\eta}).
  \end{align*}
  
  Applying  Theorem~\ref{RanPermThm} and using Lemma \ref{inducedallvar}, we complete the proof.
  The bound $\varLambda_2=O(n^{-1/3+4 \eps+\eta/3})$
  in the theorem statement follows directly from~\eqref{induced-assumption} and \eqref{bounds_PM}.
 \end{proof}
 

\begin{proof}[Proof of Corollary \ref{induced-corollary}]
 To show \eqref{induced-1}, observe that $\omega_t = O(r^{t+1})$. Therefore,
 the assumption $r^2(1+\delta^2/n) = O\(\lambda^2(1-\lambda)^2 n^{1/2+\eta}\)$implies  that
  \begin{align*}
       \frac{r^2}{2n} + \frac{(1-2\lambda) \omega_1}{ 2 \lambda(1-\lambda) n} - \frac{r^2 R}{ 2 \lambda (1-\lambda)n^2 }
       -  \frac{(1-2\lambda) R \omega_1}{ 2 \lambda^2 (1-\lambda)^2 n^2 } & = O\left(  \frac{r^2(1+\delta^2/n)}{\lambda^2 (1-\lambda)^2 n} \right) = O(n^{-1/2+\eta}),\\
  \frac{ \omega_1^2}{ 4 \lambda(1-\lambda) n^2} +       \frac{r \omega_2}{ 2 \lambda (1-\lambda)n^2 }
  + \frac{(1-2\lambda)\omega_3}{ 6\lambda^2(1-\lambda)^2 n^2}& = O\left(  \frac{r^4}{\lambda^2 (1-\lambda) ^2n^2} \right) = O(n^{-1/2+\eta}),\\
       \frac{r \omega_1 \sum_{j=1}^n (d_j-d)^3}{ 2 \lambda^2 (1-\lambda)^2 n^4 }&
       = O \left( \frac{r^3 \delta^3}{\lambda^2 (1-\lambda)^2 n^3}\right)  = O(n^{-1/2+\eta}). 
  \end{align*}
  
  For the second statement,  observe that the assumption $r = O(n^{1/3-\eps})$ implies
  that $r^2(1+\delta^2/n) = O(n^{2/3})$ and so
  \begin{align*}
    -\frac{\omega_2}{ 2\lambda(1-\lambda) n}  
    &+ \frac{R\omega_2}{ 2\lambda^2(1-\lambda)^2 n^2} = 
    O\biggl(\frac{r^3(1+\delta^2/n)}{ \lambda^2(1-\lambda)^2 n}\biggr) \\
    &= O\biggl(\frac{r}{n^{1/3}\lambda^2(1-\lambda)^2}\biggr)
     = O\biggl(\frac{n^{-\eps}}{\lambda^2(1-\lambda)^2}\biggr)
     = O(n^{-\eps/2}).
  \end{align*}
  Applying \eqref{induced-1} completes the proof.
\end{proof}

\begin{proof}[Proof of Corollary \ref{c:concentration}]
The bound on $r$ in the corollary statement implies that $r=O(\log n)$
and is equivalent to $\lambdamin^r\geq n^{-2+\eps}$.
The fact that $\E Y_n\to\infty$ thus follows from
Corollary~\ref{induced-corollary}.  In order to prove the concentration,
we use the second moment method in a standard fashion.  Define
\[ N = \binom nr \frac{r!}{\card{\Aut(H^{[r]})}}, \]
and let $H_1,\ldots,H_N$ be a list of all the potential induced copies of~$H^{[r]}$.
Let $p_{j,k}$ be the probability that both $H_j$ and $H_k$ occur simultaneously
as induced subgraphs and define
\[
  E_t = \sum_{\substack{1\leq j,k\leq N\\[0.1ex] \card{V(H_j)\cap V(H_k)}=t}} 
  \negthickspace p_{j,k},
  \qquad (0\leq t\leq r).
\]
We know that $\E Y_n^2 = \sum_{t=0}^r E_t$ and now we compare
$\E Y_n^2$ to~$(\E Y_n)^2$.
The probability $p_{j,k}$ is not provided directly by either Theorem~\ref{subcount}
or Theorem~\ref{inducedsubcount} but we can infer it from the second part
of Corollary~\ref{induced-corollary}.  By summing over all the possible subgraphs
induced by $V(H_j)\cup V(H_k)$, we find that $E_t$ asymptotically matches
the corresponding expectation for the binomial random graph model~$\calG(n,\lambda)$
to relative error $O(n^{-\eps/2}+n^{-1/2+\eta})$.
Therefore, we have
\begin{align*}
    E_0 &=   \binom nr \binom {n-r}r
    \biggl(\frac{r!}{\card{\Aut(H^{[r]})}}\biggr)^{\!\!2} \lambda^{2m}(1-\lambda)^{r(r-1)-2m}
                 \(1+O(n^{-\eps/2}+n^{-1/2+\eta})\) \\
    &= (\E Y_n)^2 \(1+O(n^{-\eps/2}+n^{-1/2+\eta})\).
\end{align*}
To bound $E_t$ from above for $t\geq 1$, we can assume that two induced copies of
$H^{[r]}$ always overlap correctly.  This gives
\begin{align*}
    E_t &\leq \binom nr \binom rt \binom {n-r}{r-t}
                 \biggl(\frac{r!}{\card{\Aut(H^{[r]})}}\biggr)^{\!\!2} 
             \lambda^{2m}(1-\lambda)^{r(r-1)-2m} \lambdamin^{-\binom t2} \\
           &\leq (\E Y_n)^2 \( 2 n^{-1} r^2 \lambdamin^{-(t-1)/2} \)^t (1+o(1)).
\end{align*}
Using the condition $\lambdamin^r\geq n^{-2+\eps}$, we have that
$\(n^{-1} 2 r^2 \lambdamin^{-(t-1)/2}\)^t=O(n^{-1/2+\eta})$ for $t=1$ and
$\(n^{-1} 2 r^2 \lambdamin^{-(t-1)/2}\)^t=O(n^{-t\eps/3})$ for $2\leq t\leq r$.
Therefore,
\[ \E Y_n^2 = (\E Y_n)^2(1 + O(n^{-\eps/2}+n^{-1/2+\eta})), \]
which implies that
\[ \Var Y_n = (\E Y_n)^2 \, O(n^{-\eps/2}+n^{-1/2+\eta}). \]
The desired result now follows from Chebyshev's Inequality. 
\end{proof}




\nicebreak


\begin{thebibliography}{99}
\itemsep=0pt

\bibitem{Barron}
E.\,N.~Barron, P.~Cardaliaguet and R.~Jensen,
Conditional essential suprema with applications,
\textit{Appl.\ Math.\ Optim.}, {\bf 48} (2003) 229--253.

\bibitem{EG1960}
P.~Erd\H{o}s, T.~Gallai,   Graphs with prescribed degrees of vertices (Hungarian), 
\textit{Matematikai Lapok}, {\bf 11} (1960) 264--274.

\bibitem{GIKM} C.~Greenhill, M.~Isaev, M.~Kwan and B.\,D.~McKay,
The average number of spanning trees in
sparse graphs with given degrees,
\textit{European J. Combin}, {\bf 63} (2017) 6--25.

\bibitem{mother} M.~Isaev and B.\,D.~McKay,
Complex martingales and asymptotic enumeration,
 \textit{Random Structures Algorithms}, {\bf 52} (2018) 617--661.

\bibitem{KSV2007} J.\,H.~Kim, B.~Sudakov, V.~Vu,
Small subgraphs of random regular graphs,
\textit{Discrete Math.}, {\bf 307} (2007),
\#15, 1961--1967.

\bibitem{KSVW}
M.~Krivelevich, B.~Sudakov, V.\,H.~Vu and N.\,C.~Wormald,
Random regular graphs of high degree,
\textit{Random Structures Algorithms}, {\bf 13} (2001) 346--363.

\bibitem{McDiarmid} C.~McDiarmid,
Concentration,
\textit{in} 
Probabilistic Methods for Algorithmic Discrete Mathematics,
\textit{Algorithms Combin.}, {\bf 16} (1998) 195--248.

\bibitem{moon} J.\,W.~Moon, \emph{Counting Labelled Trees},
Canadian Mathematical Monographs, Vol.~1, Canadian Mathematical Congress,
Montreal, 1970.

\bibitem{bdmreg} B.\,D.~McKay,
Asymptotics for symmetric 0-1 matrices with prescribed 
row sums, \textit{Ars Combin.}, {\bf 19A} (1985) 15--26.

\bibitem{ranx} B.\,D.~McKay,
Subgraphs of dense random graphs with specified degrees,
\textit{Combin. Probab. Comput.}, {\bf 20} (2011) 413--433.

\bibitem{MW90} B.\,D.~McKay and N.\,C.~Wormald,
Asymptotic enumeration by degree sequence of
graphs of high degree, \textit{European J. Combin.}, {\bf 11} (1990) 565--580.

\bibitem{XYWR2008} L.~Xiao, G.~Yan, Y.~Wu, W.~Ren,
Induced subgraph in random regular graph,
\textit{J.~Systems Science and Complexity}, {\bf 21} (2008),
\#4 645--650.

\end{thebibliography}
\end{document}